\documentclass[11pt]{article}
\pdfoutput=1
\usepackage{amsmath,amssymb,amsthm,dsfont,algorithm,algorithmic,xcolor,epstopdf,bm,graphicx}
\usepackage[margin=0.75in]{geometry}

\usepackage[normalem]{ulem}

\usepackage[caption=false]{subfig}

\newcommand{\dx}[1]{\,\mathrm{d}#1}
\newcommand{\pd}[2]{\frac{\partial #1}{\partial #2}}
\newcommand{\copyrightStatement}{Official contribution of the National Institute of Standards and Technology; not subject to copyright in the United States.}

\newtheorem{remark}{Remark}
\newtheorem{assumption}{Assumption}
\newtheorem{theorem}{Theorem}
\newtheorem{corollary}{Corollary}
\newtheorem{lemma}{Lemma}

\title{A sparsity-constrained sampling method with applications to communications and inverse scattering\thanks{\copyrightStatement}}
\author{I. Harris\thanks{Department of Mathematics, Purdue University, West Lafayette, IN
(harri814@purdue.edu)} \and J.D. Rezac\thanks{Communications Technology Laboratory, National Institute of 
Standards and Technology, Boulder, CO (jacob.rezac@nist.gov)}}
\date{\today}

\begin{document}

\maketitle

\begin{abstract}
We introduce the sparse direct sampling method (DSM) to estimate properties of a region from signals that probe the region. We demonstrate the sparse-DSM on two separate problems: estimating both the angle-of-arrival of a radio wave impinging on an array and the location and shape of an inhomogeneity from scattered acoustic waves. The sparse-DSM is qualitative in nature, so it does not require the simulation of a forward problem to solve the inverse problem. The method generalizes of two older qualitative methods, one which has low-resolution reconstructions but uses few measurements and one which is high-resolution but has higher measurement cost. The sparse-DSM inherits positive qualities from both. We demonstrate the technique on measured and simulated examples. 
\end{abstract}

\section{Introduction}
Many classic inverse problems involve estimating physical properties of a region by analyzing the dynamics of signals used to probe the area. Qualitative (or ``direct'') inversion methods solve this type of problem by exploiting analytical properties of the signals and hence avoid simulating the measurement process. In this way, they are typically much faster than iterative optimization approaches - though they also reveal less information about the unknowns and often require more high-quality measurement data \cite{CakoniColton2014,CakoniColtonHaddar2016}. In this article, we develop a new qualitative method, the \textit{sparse direct sampling method (sparse-DSM)}, which flexibly balances measurement requirements with the resolution of estimates.

The two problems we analyze with sparse-DSM here are prototypical test cases for qualitative methods. The angle-of-arrival (AOA) estimation problem (also known as the direction-of-arrival problem) is to determine the direction from which incoming radio waves impinge on an array of receiving antennas. This information is used, for example, by communications devices to electrically steer their antennae in the direction of nearby communications towers \cite{AndrewsEtAl2014, RanganEtAl2014,RohEtAl2014}. The inverse scattering problem is to estimate the location and shape of hidden obstacles by transmitting waves through a region-of-interest and measuring the resulting scattered field. This has uses in, for example, non-destructive testing and mineral prospecting \cite{CakoniColton2014}. While these are superficially different problems, we demonstrate that nearly identical qualitative inversion techniques are available for their study.

Qualitative methods have attracted much attention in the inverse scattering community since the late 1990s \cite{ColtonKirsch1996}, though simplified versions of these ideas have been present in signal processing research since at least the late 1940s \cite{Bartlett1948}. Loosely-speaking, there are two categories of qualitative methods: those based on the asymptotic behavior of post-processed data, and those based on algebraic and analytical properties of a physical model which describes the measurement data. Examples of the former class are conventional beamforming \cite{Bartlett1948, VanVeenBuckley1988} and the direct sampling method (DSM) \cite{Ito2012}. They are computationally very fast, require little measurement data, and are stable with respect to high amounts of noise. However, they often produce low-resolution estimates. We will generically refer to them as direct sampling methods. The second class contains Capon's beamformer \cite{Capon1969}, the inf-criterion \cite{Kirsch1998}, the multiple signal classification (MUSIC) algorithm \cite{Schmidt1986}, the factorization method \cite{Kirsch1998, KirschGrinberg2008}, and the linear sampling method (LSM) \cite{CakoniColton2014, ColtonKirsch1996}. These methods usually produce high-resolution estimates, but are slower, usually require a lot of measurement data, and have less stability with respect to noise. 

The sparse-DSM is a generalization of methods from the two categories which produces high-resolution results under modest measurement requirements. Indeed for AOA estimation problems, beamforming can be seen as the sparsest solution to Capon's beamformer in a sense we make clear below. The same relationship holds between the DSM and the inf-criterion in the inverse-scattering context. We transition between these classical inversion techniques by looking for solutions less sparse than beamforming (or the DSM) but still more sparse than Capon's beamformer (or the inf-criterion). We show that the sparse-DSM inherits positive qualities from both classical inversion techniques in the example inverse problems we analyze, yielding high-resolution reconstructions from limited and noisy measurement data. Since the sparse-DSM is related to classical techniques, we anticapate that it will perform well for other applications where qualitative methods have been developed, such as electromagnetic and elastic inverse scattering \cite{ArensJiLiu2020,BourgeoisLuneville2013,CakoniColtonMonk2011,CharalambopoulosEtAl2003,CharalambopoulosEtAl2007, HarrisNguyen2019,Ito2013,Nguyen2019,PourahmadianGuzinaHaddar2017} and point estimation problems in communications, seismology, and small-obstacle scattering \cite{AmariEtAl2013, Hokanson2013, PeterPlonka2013, RostThomas2002, VanTrees2002}.

Our presentation of these ideas begins by first introducing two inf-criteria in Section \ref{sec:sparseInfCrit} which will be the basis for the sparse-DSM. A subspace-constrained solution to the optimization problems in these inf-criteria reduces to simple functionals which have been used in the past as the basis for beamforming and the DSM, as described in Section \ref{sec:beamforming}. Based on this observation, we continue in Section \ref{sec:sparseDSM} to define the sparse-constrained optimization problem which we call the sparse-DSM. In Section \ref{sec:AOA}, we describe the AOA-estimation problem and how to apply sparse-DSM to estimate its solution. The same is done in Section \ref{sec:inverseScattering} for the acoustic inverse scattering problem. Based on these models we provide a collection of reconstruction results in Section \ref{sec:numerics} from both simulated and measured data that demonstrate the applicability of sparse-DSM as well as some of its positive and negative properties. The conclusion in Section \ref{sec:conclusion} outlines a 
potential future research areas related to the sparse-DSM. 

Throughout the article, we will use notation typical in linear algebra and functional analysis. For finite-dimensional objects, bold letters (e.g., $\mathbf{A}$ or $\mathbf{v}$) distinguish matrices and vectors from scalars. Caligraphic font (e.g., $\mathcal{A}$) is used to denote linear operators on infinite-dimensional spaces. The range and nullspace of an operator are indicated by $\mathcal{R}(\cdot)$ and $\mathcal{N}(\cdot)$ respectively and the orthogonal complement of a set $W$ is indicated by $W^{\perp}$. When working on an inner-product space $X$, we denote its norm by $\|\cdot\|_X$ and inner-product by $(\cdot,\cdot)_X$. The operator-norm of a linear operator mapping between spaces $X$ and $Y$ is denoted $\|\cdot\|_{X\to Y}$ while the Frobenius matrix norm is denoted $\|\cdot\|_F$. In the special case of the finite dimensional Euclidean norm, we use $\|\cdot\|_2$ and $(\cdot,\cdot)_2$. The adjoint of an operator is denoted by $\cdot^*$. Finally, normalized elements are defined by a carrot over the element, e.g., $\hat{g}:=g/\|g\|$. 

\begin{remark} 
We are not the first to join sparse optimization and qualitative inversion. As an incomplete list, see 
\cite{AlqadahEtAl2011, AlqadahEtAl2011b, AlqadahValdivia2013, Alqadah2016, ChaiMoscosoPapanicolaou2013, ChaiMoscosoPapanicolaou2014, Fannjiang2010, LeeBresler2012, KimLeeYe2012, MalioutovCetinWillsky2005}. In such studies, qualitative methods are enhanced through sparse regularization or an a priori assumption of a sparse solution. While often yielding impressive results, this differs greatly from our use of sparsity to transition between different qualitative methods.
\end{remark}

\section{Infimum criteria}\label{sec:sparseInfCrit}
We begin our discussion by presenting two similar theorems which relate the range of an operator to an optimization problem involving that operator. The two theorems differ in assumptions needed for their application, leading to small differences in the optimization problem. We will see later in the article that one of these theorems applies to point parameter reconstruction, while the other is needed for reconstructing parameters with larger support. This distinction is needed because in the AOA problem, we only look for point unknowns (the angle from which a wave impinges on an array), while in the inverse scattering problem we look for continuous unknowns (the shape and location of hidden objects). Indeed, in the simplest case, the sparse-DSM applies to measured data of the form 
\begin{equation}\label{eqn:sparseSum}
 f(x,y) = \sum_{j=1}^J \alpha_j \psi(x;z_j)\varphi(y;z_j), \qquad x\in\Gamma_x, \quad y\in\Gamma_y
\end{equation}
where $z_j\in\mathbb{R}^d$ ($d>0$) are the $J>0$ unknowns-of-interest, $\alpha_j\in\mathbb{C}$ are unknown scalars, and $\psi$ and $\varphi$ are functions whose form is known. The sets  $\Gamma_x\times\Gamma_y:=\left\{(x_m,y_n)\right\}_{m,n=1}^{M,N}\subseteq\mathbb{R}^{d_x}\times\mathbb{R}^{d_y}$  ($d_x,d_y>0$) describe a rectilinear measurement geometry. The model \eqref{eqn:sparseSum} applies to AOA estimation and inverse scattering under a single-scattering approximation as discussed in this article, and there is further a large group of applications in which a measured signal is modeled as a linear superposition of known functions with unknown parameters. An important one for applications is exponential fitting and its generalizations \cite{Hokanson2013, PeterPlonka2013}.

Note that the data matrix $[\mathbf{L}]_{m,n}=f(x_m,y_n)$ can be factored in two different ways, $\mathbf{L}=\mathbf{P}\mathbf{Q}=\mathbf{R}\mathbf{A}\mathbf{Q}$ where $[\mathbf{R}]_{m,j} = \psi(x_m;z_j)$, $[\mathbf{Q}]_{j,n} = \varphi(y_n;z_j)$, $\mathbf{A}$ is the diagonal matrix so that $[\mathbf{A}]_{j,j}=\alpha_j$, and $\mathbf{P}=\mathbf{R}\mathbf{A}$. A similar factorization is availible when the unknowns are not represented by points in space for the inverse scattering problem. In this case, we analyze the linear operator $\mathcal{L}:X\to Y$ mapping between Hilbert spaces $X$ and $Y$, and factored as either $\mathcal{L}=\mathcal{P}\mathcal{Q}$ or $\mathcal{L}=\mathcal{R}\mathcal{A}\mathcal{Q}$. Here, $\mathcal{P},\mathcal{R}:Z\to Y$, $\mathcal{Q}:X\to Z$, and $\mathcal{A}:Z\to Z$ are also linear operators and $Z$ is again a Hilbert space. The case when $\mathcal{R}=\mathcal{Q}^*$ and $X=Y$ is relevant for the applications we consider here.

\subsection{Two infimum criteria}
As a first step towards developing the sparse-DSM, we relate a factorization of measured data into a product of linear operators to a functional which can be calculated from measurement data. Note that the remainder of this Section is included primarily for completeness and to introduce notation and a reader familiar with infimum-criteria may move to Section \ref{sec:sparseDSM} where the sparse-DSM is introduced. Indeed, while Theorem \ref{thm:boundedInfCrit} and Corollary \ref{cor:pointEstimationInf} are new to the best of our knowledge, they are similar to Theorem \ref{thm:coerciveInfCrit} and Corollary \ref{cor:extendedEstimationInf} which we adopt from \cite{Kirsch1998}. Each of these relates the range of an operator to an optimization problem; the range of this operator will be shown to be related to the unknowns of our applications of interest in Sections \ref{sec:AOA} and \ref{sec:inverseScattering} - thus relating unknowns to an optimization problem which can be solved to estimate the unknowns.

\begin{theorem}\label{thm:boundedInfCrit}
Let $X,$ $Y,$ and $Z$ be Hilbert spaces and $\mathcal{L}:X\to Y$ be a bounded linear operator. Assume $\mathcal{L}$ can be factored as 
$\mathcal{L}=\mathcal{P}\mathcal{Q}$ where $\mathcal{P}:Z\to Y$ and $\mathcal{Q}: X\to Z$ are bounded and linear operators. Assume $\mathcal{P}$ is also bounded below in the sense that there exists a constant $c>0$ so that $c\|z\|_Z \leq \|\mathcal{P}z\|_Y$ for every $z\in Z$. Then, for any $\varphi\in X$, 
\begin{equation*}
\varphi\in\mathcal{R}(\mathcal{Q}^*) \qquad \text{ if and only if } \qquad \mathcal{I}^{(0)}(\varphi):= \inf_{g\in X} \left\{ \|\mathcal{L}g\|_Y \, : \, (g,\varphi)_X=1\right\} > 0. 
\end{equation*}
\end{theorem}
\begin{proof}
Since $\mathcal{P}$ is bounded below and continuous, the factorization of $\mathcal{L}$ shows
\begin{equation}\label{eqn:lowerBound}
c\|\mathcal{Q}g\|_Z \leq \|\mathcal{L} g\|_Y \leq \|\mathcal{P}\|_{Z\to Y}\|\mathcal{Q}g\|_Z, \quad \forall g\in X.
\end{equation}

Assume that $\varphi\in\mathcal{R}(\mathcal{Q}^*)$ with some nonzero $\phi\in Z$ so that $\mathcal{Q}^*\phi = \varphi$. Then any $g\in X$ satisfying $(g,\varphi)_X=1$ also satisfies $(\mathcal{Q}g,\phi)_Z=1$. Applying \ref{eqn:lowerBound} to such a $g$ (say $g=\varphi/\|\varphi\|^2_X$) demonstrates that
\begin{equation*}
\|\mathcal{L}g\|_Y \geq c\|\mathcal{Q}g\|_Z = \frac{c}{\|\phi\|_Z}\|\mathcal{Q}g\|_Z\|\phi\|_Z \geq \frac{c}{\|\phi\|_Z} \left|\left(\mathcal{Q}g,\phi\right)_Z\right| = \frac{c}{\|\phi\|_Z}>0. 
\end{equation*}

On the other hand, assume $\varphi\notin \mathcal{R}(\mathcal{Q}^*)$. We wish to demonstrate that $\mathcal{I}(\varphi)=0$. From the upper bound in \eqref{eqn:lowerBound}, it is sufficient to show there exists a sequence $g_n\in X$ with $(g_n,\varphi)_X=1$ and $\|\mathcal{Q}g_n\|_Z\to 0$ as $n\to\infty$. Equivalently, define $h_n =\hat{\varphi}-g_n$ and the  the set $\mathcal{P}_{\varphi} := \left\{\psi\in X \, : \, \left(\psi,\varphi\right)_X = 0\right\}$ and demonstrate the existence of $h_n\in \mathcal{P}_{\varphi}$ such that $\|\mathcal{Q}( \hat{\varphi}-h_n)\|_Z\to 0$. 

Since $\mathcal{Q}\hat{\varphi}\in\mathcal{R}(\mathcal{Q})$, the result follows by demonstrating that $\mathcal{Q}(\mathcal{P}_{\varphi})$ is dense in $\mathcal{R}(\mathcal{Q})$ - i.e., $\mathcal{Q}(\mathcal{P}_{\varphi})^\perp = \mathcal{R}(\mathcal{Q})^{\perp}$. Indeed, 
\begin{equation*}
\phi \in \mathcal{Q}(\mathcal{P}_{\varphi})^\perp \quad \text{ if and only if } \quad \left(\mathcal{Q}^*\phi, \psi\right)_X = 0 \quad \text{ for all } \quad \psi\in \mathcal{P}_\varphi. 
\end{equation*}
But this is equivalent to $\mathcal{Q}^*\phi\in\text{span}\{\varphi\}$. Since $\varphi\notin\mathcal{R}(\mathcal{Q})$, it follows that $\mathcal{Q}^*\phi = 0$. The equality $\mathcal{N}(\mathcal{Q}^*)=\mathcal{R}(\mathcal{Q})^\perp$ completes the proof.  
\end{proof}

\begin{remark}
The reader may note that, when $\mathcal{L}\mathcal{L}^*$ is invertible, $\mathcal{I}^{(0)}(\varphi)$ has the closed-form minimum
\begin{equation*}
g(y)  = \frac{\left(\mathcal{L}\mathcal{L}^*\right)^{-1}\varphi(y)}{\varphi(y)^*\left(\mathcal{L}\mathcal{L}^*\right)^{-1}\varphi(y)}. 
\end{equation*}
When expressed in terms of the AOA problem, this is typically called Capon's method \cite{Capon1969}. Unfortunately, in practice $\mathcal{L}$ is polluted by noise and some  form of generalized inversion is needed to find $g$. The sparse-DSM discussed in this paper can be considered a way to address this. 
\end{remark}

The assumption Theorem \ref{thm:boundedInfCrit} that $\mathcal{P}$ is bounded below is relatively restrictive. Indeed, a bounded linear operator acting between Hilbert spaces is bounded below if and only if it is injective with closed range. This will present difficulties for the inverse scattering problem, and we instead rely on the inf-criterion introduced in \cite{Kirsch1998}. We modify it here to better match the notation in Theorem \ref{thm:boundedInfCrit}. Note that the assumptions on $\mathcal{L}$ and the indicator functional require modification from Theorem \ref{thm:boundedInfCrit}. 

\begin{theorem}[Infimum-Criterion \cite{Kirsch1998}]\label{thm:coerciveInfCrit}
Let $\mathcal{L}=\mathcal{Q}^*\mathcal{A}\mathcal{Q}$ be a bounded linear operator mapping between a Hilbert space $X$ and itself, where $\mathcal{Q}: X\to Z$ and $\mathcal{A}:Z\to Z$ are bounded linear operators and $Z$ is another Hilbert space. Assume that $\mathcal{A}$ is coercive in the sense that there is a constant $c> 0$ such that for every $z\in \mathcal{R}(\mathcal{Q})\subseteq Z$, 
\begin{equation}\label{eqn:coerciveAssump}
c\|z\|^2_Z \leq |(z,Az)_Z|. 
\end{equation}
Then, for any non-zero $\varphi\in X$,
\begin{equation*}
\varphi\in\mathcal{R}(\mathcal{Q}^*)  \qquad \text{ if and only if } \qquad \mathcal{I}^{(1)}(\varphi):=\inf_{g\in X} \left\{|(g,\mathcal{L}g)_X|^2\,:\, (g,\varphi)_X=1 \right\}>0. 
\end{equation*}
\end{theorem}

\subsection{Relationship between inf-criteria and unknowns}

We wish to relate unknowns in each inverse problem to the optimization problems in Theorems \ref{thm:boundedInfCrit} and \ref{thm:coerciveInfCrit}. This is straightforward when measurements are of the form \eqref{eqn:sparseSum}. 

\begin{corollary}[A min-criterion for point estimation]\label{cor:pointEstimationInf} 
Let $\mathbf{L}\in\mathbb{C}^{M\times N}$ be a matrix which can be factored as $\mathbf{L}=\mathbf{P}\mathbf{Q}$ with $\mathbf{P}\in\mathbb{C}^{M\times J}$ and $\mathbf{Q}\in\mathbb{C}^{J\times N}$.
\begin{enumerate}
\item Assume $\mathbf{P}$ is injective. Then for any vector $\bm{\varphi}\in\mathbb{C}^{N\times 1}$,
\begin{equation*}
\bm{\varphi}\in\mathcal{R}(\mathbf{Q}^*) \quad \text{ if and only if } \quad \mathbf{I}^{(0)}(\bm{\varphi}):=\min_{\mathbf{g}\in\mathbb{C}^{N\times 1}}\left\{\|\mathbf{L}\mathbf{g}\|^2_2  \,:\, (\mathbf{g},\bm{\varphi})_2=1\right\} >0. 
\end{equation*}
\item Assume data is measured according to model \eqref{eqn:sparseSum} so that $[\mathbf{L}]_{m,n} = f(x_m,y_n)$, $[\mathbf{P}]_{m,j} = \bm{\psi}(x_m;z_j)$, and $[\mathbf{Q}]_{j,n} = \alpha_j\bm{\varphi}(y_n;z_j)$. If $\{\bm{\psi}(x,z_j), x\in\Gamma_x\}_{j=1}^J$ and $\{\overline{\bm{\varphi}(y,z_j)}, y\in\Gamma_y\}_{j=1}^J$ are linearly independent as functions of $z_j$ then for any $z\in\mathbb{R}^d$,
\begin{equation*}
z=z_j \quad \text{ if and only if } \quad \mathbf{I}^{(0)}\left(\overline{\bm{\varphi}(\cdot;z)}\right)>0.
\end{equation*}
\end{enumerate}
\end{corollary}

\begin{proof}
1. We may replace the infimum of $\mathcal{I}^{(0)}$ in Theorem \ref{thm:boundedInfCrit} with a minimum here because $\mathbf{I}^{(0)}(\bm{\varphi})$ is a linear least squares problem. Indeed, any $\mathbf{g}$ which satisfies the constraint $(\mathbf{g},\bm{\varphi})_2=1$ can be decomposed into $\mathbf{g}=\hat{\bm{\varphi}}+\mathbf{f}$ where $\mathbf{f}$ is orthogonal to $\bm{\varphi}$. Hence, if $\bm{\Pi}_{\varphi}$ is the matrix which orthogonally projects onto $\text{span}\{\bm{\varphi}\}^\perp$ then $\|\mathbf{L}\mathbf{g}\|^2_2=\|\mathbf{L}\bm{\Pi}_{\varphi}\mathbf{h} + \mathbf{L}\hat{\bm{\varphi}}\|^2_2$. The infimum problem in $\mathcal{I}^{(0)}$ is hence a linear least squares problem for $h$ which always obtains its solution.

The corollary then follows immediately from Theorem \ref{thm:boundedInfCrit} because $\mathbf{P}$ is bounded below: when (finite-dimensional) $\mathbf{P}$ is injective, $\left(\mathbf{P}^*\mathbf{P}\right)^{-1}$ exists and hence, for any $\mathbf{h}\in\mathbb{C}^{J\times 1}$, 
\begin{equation*}
\|\mathbf{h}\|_2 = \|\left(\mathbf{P}^*\mathbf{P}\right)^{-1}\mathbf{P}^*\mathbf{P}\mathbf{h}\|_2 \leq \|\left(\mathbf{P}^*\mathbf{P}\right)^{-1}\mathbf{P}^*\|_F \|\mathbf{P}\mathbf{h}\|_2.
\end{equation*}

2. The linear independence assumption on $\bm{\psi}$ gives that $\mathbf{P}$ is injective. Hence, part 1 of this corollary gives that for any $\bm{\varphi}\in\mathbb{C}^{N\times 1}$, $\bm{\varphi}\in\mathcal{R}(\mathbf{Q}^*)$ if and only if $\mathbf{I}^{(0)}(\bm{\varphi})>0$. If $z=z_j$, then $\overline{\bm{\varphi}(\cdot;z_j)}\in\mathcal{R}(\mathbf{Q}^*)$ because it is proportional to one of the columns of $\mathbf{Q}^*$, which are linearly-independent by assumption. On the other hand, if $\overline{\bm{\varphi}(\cdot,z)}\in\mathcal{R}(\mathbf{Q}^*)$ then $\overline{\bm{\varphi}(y;z)} = \sum_{j=1}^J \overline{\alpha}_j \overline{\bm{\varphi}}(y;z_j)$ for all $y\in\Gamma_y$. By linear independence and the assumption that $\alpha_j\neq 0$ for each $j$, this is impossible unless $z=z_j$. 
\end{proof}
{

We must defer to Section \ref{sec:inverseScattering} to prove an analogous result for the inverse scattering problem of identifying a set $D\subseteq\mathbb{R}^d$ describing the shape and location of hidden scatterers from a measured scattered field. In that section, we will build the necessary theory to prove the hypothesis to the following corollary to Theorem \ref{thm:coerciveInfCrit}.

\begin{corollary}\label{cor:extendedEstimationInf}
Make the same assumptions as in Theorem \ref{thm:coerciveInfCrit}. If there is a function $\varphi(\cdot;z)\in X$ parameterized by $z\in\mathbb{R}^d$ such that 
\begin{equation*}
z\in D \quad \text{ if and only if } \quad \varphi(\cdot;z)\in\mathcal{R}(\mathcal{Q}^*)
\end{equation*}
then 
\begin{equation*}
z\in D \quad \text{ if and only if } \quad \mathcal{I}^{(1)}(z):= \inf_{g\in X}\left\{\left|\left(g,\mathcal{L}g\right)_X\right|^2\,:\,\left(g,\varphi(\cdot;z)\right)_X=1\right\}>0. 
\end{equation*}
\end{corollary}

While the theoretical justification of the two infimum-criteria requires different loss functions, in practice we treat the optimization problems in the same way. This is because under certain circumstances their discretized forms are equivalent.
\begin{lemma}[Adapted from \cite{BoydVandenberghe2004}]\label{lemma:equiv}
Assume $\bm{\Pi}_\varphi^* \mathbf{L} \bm{\Pi}_\varphi$ is positive semi-definite where $\bm{\Pi}_{\varphi}$ is the projection operator used above. Then the vector $\mathbf{g}\in\mathbb{C}^{M\times 1}$ is an optimal solution to $\mathbf{I}^{(1)}(z)$ if and only if $\bm{\Pi}_\varphi^*\mathbf{L}^*\bm{\Pi}_{\varphi} \mathbf{h} = -\bm{\Pi}_{\varphi}^*\mathbf{L}^*\hat{\bm{\varphi}}(\cdot;z)$.
\end{lemma}
As such, as long as $\bm{\Pi}_\varphi^* \mathbf{L}\bm{\Pi}_{\varphi}$ is semi-positive definite, both optimization problems can be phrased as least-squares problems. While this semi-positive definite assumption may not always hold (for reasons related to noise or location of measurements, for example), this does not pose a problem in results we show in Section \ref{sec:numerics}.

\section{The sparse-direct sampling method}\label{sec:sparseDSM}
The inf-criteria was one of the earliest qualitative methods developed in scattering theory. Nonetheless, it has received less study than other similar techniques such as the LSM and factorization method. Part of the reason for this is the comparitively-slow computational speed required to solve the optimization problem in the inf-criterion. More recent work \cite{BorceaMeng2019,Liu2017} has used the inf-criterion and related factorization method to develop theoretical guarantees about the performance of DSMs which are fast and stable with respect to measurement quality. In a similar manner, we show in Section \ref{sec:beamforming} that beamforming and DSM algorithms can be seen as (very) constrained solutions of the optimization problems in infimum-criteria. Using this idea, we develop the sparse-DSM, which is a generalization of both the inf-criterion and the DSMs discussed here. The idea behind the sparse-DSM is to again constrain solutions of the infimum-criteria optimization problems - but to adapt the constraint to the measurement data. Intuitively, sparse-DSM is parameterized so that at the most constrained, it reduces to DSM or beamforming and at the least constrained it reduces to a full infimum-criterion.

\begin{remark}It is somewhat inconvenient to refer separately to the optimization problems posed in Theorems \ref{thm:boundedInfCrit} and \ref{thm:coerciveInfCrit}. As such, throughout the rest of the article we use the notation 
\begin{equation*}
\mathcal{K}^{(i)}g :=\begin{cases} \|\mathcal{L}g\|_X & i=0 \\ \left|\left(\mathcal{L}^*g,g\right)_X\right|^2 & i=1 \end{cases}
\end{equation*}
to simplify the simultaneous discussion. The matrix $\mathbf{K}^{(i)}$ denotes the equivalent finite-dimensional operator.\end{remark} 

\subsection{Beamforming and Direct Sampling Methods}\label{sec:beamforming}
Solutions to the optimization problems in Corollaries \ref{cor:pointEstimationInf} and \ref{cor:extendedEstimationInf} can be written in the form $\hat{\varphi} + \Pi_\varphi h$
where $\Pi_\varphi$ projects onto $\text{span}\{\varphi\}^{\perp}$ and $h$ is some element in its domain. By taking $h\equiv 0$, it follows immediately from Corollaries \ref{cor:pointEstimationInf} and \ref{cor:extendedEstimationInf} that if $\varphi\in\mathcal{R}(\mathcal{Q}^*)$ then 
\begin{equation*}
\mathcal{I}_{\emptyset}^{(i)}(\varphi):=\mathcal{K}^{(i)}\hat{\varphi}>0, \quad i=0,1.
\end{equation*}
Relating this to the unknowns of each inverse problem, Corollary \ref{cor:pointEstimationInf} yields that if $z=z_j$ then $\mathcal{I}_{\emptyset}^{(0)}(z)>0$. In the same way, Corollary \ref{cor:extendedEstimationInf} gives that if $z\in D$ then $\mathcal{I}_{\emptyset}^{(1)}(z)>0$. This is half of the principle behind the DSM and beamforming
described here and elsewhere\footnote{Early ideas for DSMs in inverse scattering \cite{Ito2012, Ito2013, Potthast2010} primarily considered the case of one incident direction. The DSMs introduced in those articles do not include a norm or inner-product because there is only one point over which to take those operations. Further studies \cite{LeemLiuPelekanos2018,Liu2017} have studied DSMs with more data and derived similar indicator functions as we do here.}: evaluating $\mathcal{I}_{\emptyset}^{(i)}(z)$ on a grid $z\in\mathcal{Z}\subseteq\mathbb{R}^d$ containing the true unknowns will positively identify all unknowns-of-interest. However, unlike in the infimum-criteria described above, we cannot say anything about false positives: indeed, as we will simulate below, DSMs often blur together nearby unknowns to make them look like a single unknown with no separation. Such a false positive is a hallmark of the sometimes low-resolution reconstructions produced by the class of DSMs we consider in the paper.

The second half of the justification of DSMs and beamforming algorithms is a demonstration that $\mathcal{I}_{\emptyset}^{(i)}(z)$ has a local maximum only at the unknowns-of-interest, $z=z_j$ or $z\in D$. We will show such behavior for the AOA problem and inverse scattering problem in the sections below. Unfortunately, we also show that if high-resolution reconstructions are required (for example, if two point unknowns are very close together) then we may not be able to distinguish them. Indeed, beamforming and DSM algorithms have significant positive and negative properties. On one hand, they have incomplete theoretical justification and low-resolution reconstructions. On the other hand, evaluating $\mathcal{I}_{\emptyset}^{(i)}(z)$ is significantly faster than solving an optimization problem. Moreover, numerical evidence suggests that for practical problems, there are easier-to-achieve measurement and noise requirements for these techniques than, e.g., infimum-criteria \cite{Ito2012, Ito2013}. The sparse-DSM is one way to balance these positive and negative aspects. 

\subsection{Derivation and properties of the sparse-DSM}
We call the technique introduced in the article the \textit{sparse}-DSM because in practice, the parameterization is accomplished by solving a \textit{sparse}-constrained least-squares problem described below. The sparse least squares problem is related to unknowns through a corollary to Theorems \ref{thm:boundedInfCrit} and \ref{thm:coerciveInfCrit}.

\begin{corollary}\label{cor:pointEstimationSparseInf}
Make the same assumptions as Theorem \ref{thm:boundedInfCrit} (or Theorem \ref{thm:coerciveInfCrit}). Then, for $i=0$ (or $i=1$) and $\varphi\in X$,
\begin{equation*}
\varphi\in\mathcal{R}(\mathcal{Q}^*) \text{ if and only if } \mathcal{I}_\chi^{(i)}(\varphi):=\inf_{h\in\chi}\left\{\mathcal{K}^{(i)}( \hat{\varphi} + h)\right\}>0  
\end{equation*}
 for all $\chi\subseteq \mathcal{P}_\varphi:=\{h\in X \, | \, (h,\varphi)_X = 0\}$.
\end{corollary}
\begin{proof}
We only consider the adaptation from Theorem \ref{thm:boundedInfCrit}; the proof is identical when adapting Theorem \ref{thm:coerciveInfCrit}.

It is immediate from Theorem \ref{thm:boundedInfCrit} that $\mathcal{I}_{\chi}^{(i)}(\varphi)>0$ for all $\chi\subseteq\mathcal{P}_\varphi$ implies $\varphi\in\mathcal{R}(\mathcal{Q}^*)$ since it includes the case $\chi=\mathcal{P}_{\varphi}$. On the other hand, if $\varphi\in\mathcal{R}(\mathcal{Q}^*)$ then $\mathcal{I}^{(i)}_{\mathcal{P}_\varphi}(\varphi)>0$. Since $\mathcal{I}^{(i)}_{\chi}(\varphi)\geq \mathcal{I}^{(i)}_{\mathcal{P}_\varphi}(\varphi)$ for any $\chi\subseteq \mathcal{P}_\varphi$, we see that $\varphi\in\mathcal{R}(\mathcal{Q}^*)$ implies $\mathcal{I}^{(i)}_{\chi}(\varphi)>0$ for all $\chi\subseteq\mathcal{P}_\varphi$. 
\end{proof}

Continuing to use the same language as Corollary \ref{cor:pointEstimationSparseInf}, \textit{the sparse-DSM entails fixing a strictly-smaller subset $\hat{\chi}\subset \mathcal{P}_\varphi$ and calculating $\mathcal{I}_{\hat{\chi}}^{(i)}(\varphi)$}. Ultimately, we relate the unknown-of-interest (e.g., $\{z_j\}_{j=1}^J$ or $D$) to the sign of $\mathcal{I}_{\hat{\chi}}^{(i)}(z)$ evaluated on a grid $\mathcal{Z}\subset\mathbb{R}^{d}$ containing the unknowns. It may not be immediately obvious why this more complicated optimization procedure with fixed $\hat{\chi}\subset\mathcal{P}_\varphi$ offers any improvement over, for example, a standard DSM. Indeed, Corollary \ref{cor:pointEstimationSparseInf} reveals that $\varphi\in\mathcal{R}(\mathcal{Q}^*)$ implies $\mathcal{I}_{\hat{\chi}}^{(i)}(\varphi)>0$ but that $\mathcal{I}_{\hat{\chi}}^{(i)}(\varphi)>0$ does not necessarily imply anything about $\mathcal{R}(\mathcal{Q}^*)$. In other words, the sparse-DSM is guaranteed to reveal the location of all unknowns, but has the potential for false positives -- the same problem as standard DSM and beamforming. 

The potential for improvement of sparse-DSM over other techniques becomes more clear by defining explicitly a false positive from a standard DSM algorithm. Indeed, noticing that the standard DSM corresponds to the ``optimization problem" $\mathcal{I}_{\hat{\chi}}^{(i)}(\varphi)$ with $\hat{\chi}=\emptyset$, a false positive occurs when $\mathcal{K}^{(i)}\hat{\varphi}>0$ but $\inf_{h\in\mathcal{P}_{\varphi}}\left\{\mathcal{K}^{(i)}( \hat{\varphi} + h)\right\}=0$.
As the difference in area between $\emptyset$ and $\mathcal{P}_\varphi$ is substantial, there are many points in the sets at which this could occur. Intuitively-speaking, increasing the area of $\hat{\chi}$ ($\emptyset\subset\hat{\chi}\subset\mathcal{P}_\varphi$) reduces the chance that $\mathcal{I}_{\hat{\chi}}^{(i)}(\varphi)>0$ but $\mathcal{I}^{(i)}(\varphi)=0$ simply because the size difference between $\hat{\chi}$ and $\mathcal{P}_\varphi$ is smaller. This is why we expect the sparse-DSM to improve over standard DSM. Indeed, while these observations are not precise, the results shown in Section \ref{sec:numerics} validate them on the problems we consider here.

\subsection{Implementation}
In practice, we implement the sparse-DSM by solving the sparse-constrained minimization problems
\begin{equation}\label{eqn:sparseDSM}
\mathbf{I}^{(i)}_k(z) := \min_{\mathbf{h}\in\mathbb{C}^{N\times 1}} \left\{\mathbf{K}^{(i)}(\hat{\bm{\varphi}}+\bm{\Pi}_{\varphi}\mathbf{h}) \, : \, \|\mathbf{h}\|_0\leq k \right\}, \quad (i=0,1), \quad z\in\mathcal{Z}
\end{equation}
where $\bm{\Pi}_\varphi$ is again the projection matrix onto span$\{\varphi\}^{\perp}$, $\mathcal{Z}$ is a set containing the true unknowns, and $\|\cdot\|_0$ is the semi-norm which indicates the number of non-zero entries in a vector. We call $I_k^{(i)}(z)$ the $k$-DSM. The product $\bm{\Pi}_\varphi \mathbf{h}$ above can be thought of as mapping to specific subsets of $\mathcal{P}_\varphi$. Hence, the minimization problem in \eqref{eqn:sparseDSM} is a version of
the minimization problem in Corollary \ref{cor:pointEstimationSparseInf}; increasing the value of $k$ leads to a larger $\chi\subset\mathcal{P}_\varphi$ over which the minimization takes place. As in the discussion above about Corollary \ref{cor:pointEstimationSparseInf}, $k=0$ reduces to a standard DSM and increasing $k$ reduces the chance of false positives while the technique becomes an inf-criterion. However, too-small of a $k$ still can lead to false positives. While we do not provide a method to automate the choice of $k$ in this article, we remark $k$ is independent of the unknowns themselves.

In addition to the sparse constrained indicator functions $\mathbf{I}^{(i)}_k(z)$ we will consider a mismatch constrained problem which we call the error-DSM
\begin{equation}\label{eqn:errorDSM}
\mathbf{I}^{(i)}_\delta(z) := \min_{\mathbf{h}\in\mathbb{C}^{N\times 1}} \{ \|\mathbf{h}\|_0\,:\, \mathbf{K}^{(i)}(\bm{\Pi}_\varphi \mathbf{h} + \hat{\bm{\varphi}}(\cdot;z)) \leq \delta\}\quad(i=0,1). 
\end{equation}
In principle, $\delta$ can be chosen as a function of $z$, though we do not take full advantage of such flexibility here. 

There has been a large amount of work in the past 20 years on efficient solutions to sparse estimation problems of the form \eqref{eqn:sparseDSM} and \eqref{eqn:errorDSM}. Since a benefit of qualitative methods is their speed, we require a fast solution technique. Indeed, we must solve an equation of this form repeatedly for different values of $z\in\mathcal{Z}$. In the results given in Section \ref{sec:numerics}, we use a fast solution technique based on the batch orthogonal matching pursuit (batch-OMP) method of \cite{RubinsteinEtAl2008}, which was developed to quickly solve problems of the form \eqref{eqn:sparseDSM} and \eqref{eqn:errorDSM}. Batch-OMP takes a greedy approach to solving sparse least squares problems and reduces computational overhead by computing certain matrix-matrix and matrix-vector products only once. Nonetheless, any reliable sparse optimization routine should produce similar results to what we show here. We refer the reader to \cite{RubinsteinEtAl2008} for full details. 

\section{The Direction-of-Arrival Problem}\label{sec:AOA}
In the AOA problem, $J$ signals propagating through space are measured by an array of $M$ receivers. The signals are a result of 
some transmitting antennas outputting a wave which propagates to the receivers. We assume that the signals which impinge on the array of receivers are in 
the far field of the array and are narrowband with wavenumber $k=\frac{2\pi}{\lambda}$ where $\lambda$ is the wavelength of the incident wave. Assume also that the transmitters and receivers in this problem are isotropic, which allows us to ignore antenna specifics. This simplified model is typically used in AOA estimation research, though some modifications may be required in practice. 

Denote the location of the $m$th antenna in the receiving array by $x_m\in\mathbb{R}^d$ ($d>0$) and the time at which the signal is recorded by $t_\ell$, $\ell=1,\ldots, L$. Assume each incoming wave, indexed by $j=1,\ldots,J$ impinges on this array from a direction on the surface of a unit ball, $z_j\in\mathbb{S}^2$. In the problems we consider here, the array is a 2-dimensional rectangle with $M$ antennae which are uniformly distributed in each direction, $x_m = (m_x\Delta x, m_y\Delta y,0)$ where $m_x=1,\ldots,M_x$ and $m_y=1,\ldots,M_y$ are indices such that $M=M_xM_y$ and $\Delta_x$ and $\Delta_y$ are the uniform distance between successive array elements in the $x$ and $y$ dimensions, respectively. For simplicity we center the z-coordinate at $0$. To simplify notation, we convert to spherical coordinates $(x,y,z)=(\rho \sin{(\theta)}\cos{(\phi)}, \rho \sin{(\theta)}\sin{(\phi)}, \rho \cos{(\theta)})$
 for $\rho\geq 0$,  $\theta\in [0,\pi]$, and $\phi\in [0,2\pi)$. The unknown AOAs $z_j\in\mathbb{S}^2$ can be parameterized by $(\theta_j,\phi_j)$, and for notational ease we introduce the notation $u_j = \sin(\theta_j)\cos(\phi_j)$ and $v_j = \sin(\theta_j)\sin(\phi_j)$ and  $\varphi((x,y);(u,v)) := \exp{(-ik\left(xu+yv\right))}$. 

Under these assumptions, the AOA signal model is  \cite{VanTrees2002}
\begin{equation*}
f(x_m,t_\ell) = \sum_{j=1}^J s_j(t_\ell)\varphi((m_x\Delta x,m_y \Delta y),(u_j,v_j)) \quad m=1,\ldots,M
\end{equation*}
where $s_j(t_\ell)$ is the time-domain component of the measured signal at time $t=t_\ell, \ell=1,\ldots,L$.
Hence, the AOA estimation problem is to: \textit{Estimate $(u_j,v_j)\in [-1,1]^2$ from the measurements $f(x_m,t_\ell)$}. 

Below we apply sparse-DSM to data collected in the data matrix $\mathbf{L}\in\mathbb{C}^{L\times M}$ defined by $[\mathbf{L}]_{\ell,(m_x,m_y)}=f((x_{m_x},x_{m_y}),t_{\ell})$. This is a discretization of the operator $(\mathcal{L}g)(t) = \int_{\Gamma_x\times \Gamma_y} f(x,t)g(x)\dx{x}$. The data matrix can be factored as $\mathbf{L}=\mathbf{P}\mathbf{Q}$ with $[\mathbf{P}]_{\ell,j} = s_j(t_\ell)$ and $[\mathbf{Q}]_{j,(m_x,m_y)} = \varphi((m_x\Delta_x,m_y\Delta_y);(u_j,v_j))$. Corollary \ref{cor:pointEstimationInf} provides a technique to estimate $(u_j,v_j)$ from $\mathbf{L}$. 
\begin{corollary}\label{cor:AOAInf}
Assume the $J$ vectors $\{s_j(t_\ell), \ell=1,\ldots,L\}_{j=1}^J$ are linearly independent. Then
\begin{equation*}
(u,v)=(u_j,v_j) \quad \text{ if and only if } \quad \min_{\mathbf{g}\in\mathbb{R}^{M\times 1}} \left\{\|\mathbf{L}\mathbf{g}\|_2^2 \, : \, (\mathbf{g},\overline{\bm{\varphi}((\cdot,\cdot);(u,v))})_2=1\right\}>0. 
\end{equation*}
\end{corollary}
\begin{proof}
This follows immediately from Corollary \ref{cor:pointEstimationInf} because $\varphi(m_x\Delta x,m_y\Delta y;u,v)$ are linearly independent for different $(u,v)$ values. 
\end{proof}

\begin{remark}\label{remark:LI}
When $s_j(t)$ are not linearly-independent for different $j$ - for instance in the common situation when a reflection causes the same time-domain signal to impinge on the array from multiple directions - the estimation techniques discussed here lose their theoretical backing. This is commonly called the correlated signal problem in the AOA literature and there are many methods for overcoming it \cite{PillaiKwon1989, ShanWaxKailath1985, ZiskindWax1988}. 
\end{remark}

\subsection{Beamforming}
A beamformer can be constructed for this problem by following the discussion in Section \ref{sec:beamforming} to Corollary \ref{cor:AOAInf}. Indeed, this discussion reveals that the indicator 
\begin{equation*}
\mathbf{I}_0^{(0)}((u,v)):=\|\mathbf{L}\overline{\varphi((\cdot,\cdot);(u,v))}\|_2^2
\end{equation*}
is positive at $(u,v)=(u_j,v_j)$, $j=1,\ldots,J$. More work is required to determine the behavior of $\mathbf{I}_0^{(0)}((u,v))$ away from $(u_j,v_j)$. To this end, note that
\begin{equation*}
[\mathbf{L}\overline{\varphi((\cdot,\cdot);(u,v))}]_\ell = \sum_{j=1}^J s_j(t_\ell) D_{M_x+1}(k\Delta x (u_j - u)) D_{M_y+1}(k\Delta y (v_j - v))
\end{equation*}
where $D_N$ is the Dirichlet kernel 
\begin{equation*}
D_N(x) := \sum_{k=0}^{N-1} \exp{(ikx)} = \exp{(i(N-1)x/2)}\frac{\sin{(Nx/2)}}{\sin{(x/2)}}.
\end{equation*}
In the case that $J=1$, a straightforward but lengthy calculation demonstrates that $\mathbf{I}_0^{(0)}(u,v)$ has a local maximum at $(u,v)=(u_1,v_1)$. For $J>1$ the calculations become more difficult so we show the expected behavior of beamforming with simulation. Figure \ref{fig:DirichletKernel} shows the combination of Dirichlet kernels $\left|D_N(x)+D_N(x-\alpha)\right|$ for $N=50$, $\alpha=\{1/16,1/8,1/4,1/2\}$ over $x\in[-\pi,\pi]$. Here, $\alpha$ is analogous to the angle of one incomming wave impinging on an array of receivers and as the figure demonstrates, the Dirichlet kernel peaks at $x=0$ and $x=\alpha$.  Indeed, up to scaling, this figure mimics beamforming 1-dimensional AOAs located at $x=0$ and $x=\alpha$. The figure also demonstrates that the resolution of beamforming is not expected to be too high - when $\alpha=1/16$, there is no obvious separation between the peaks. Indeed, this is a manifestation of the false positive concern about beamforming we discuss above: when $\mathbf{I}_0^{(0)}(u,v)$ is positive for $(u,v)$ which are not true AOAs, we are not able to visually distinguish between true AOAs close to each other. We refer to the extensive research on beamforming in the communications literature (\cite{VanVeenBuckley1988} and references therein) for more details and a deeper discussion of this technique, including ways to improve the resolution of reconstructions.

\begin{figure}[tbhp]
\centering
\subfloat[$\alpha = 1/2$]{\includegraphics[height=1in]{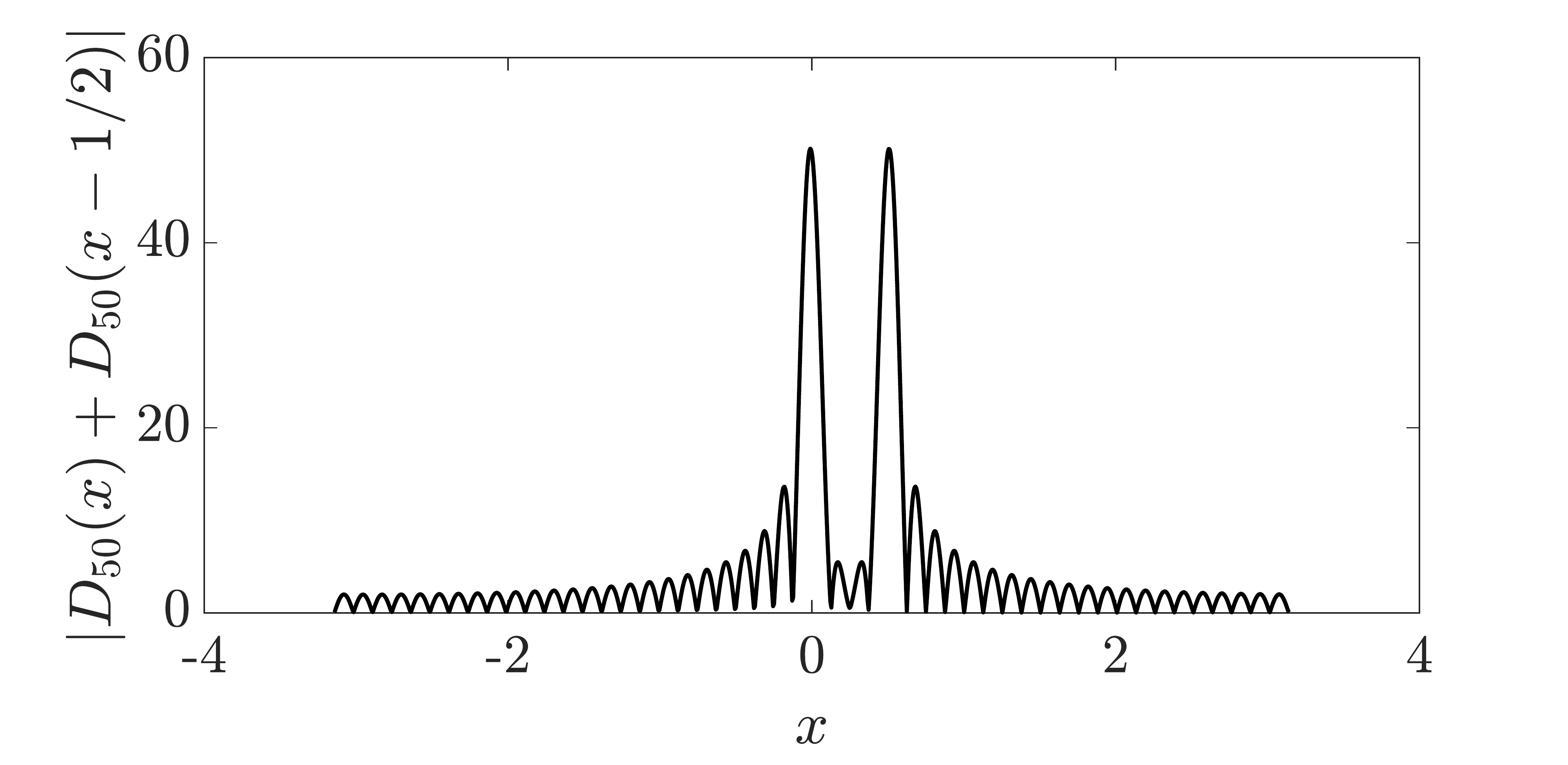}}
\subfloat[$\alpha = 1/4$]{\includegraphics[height=1in]{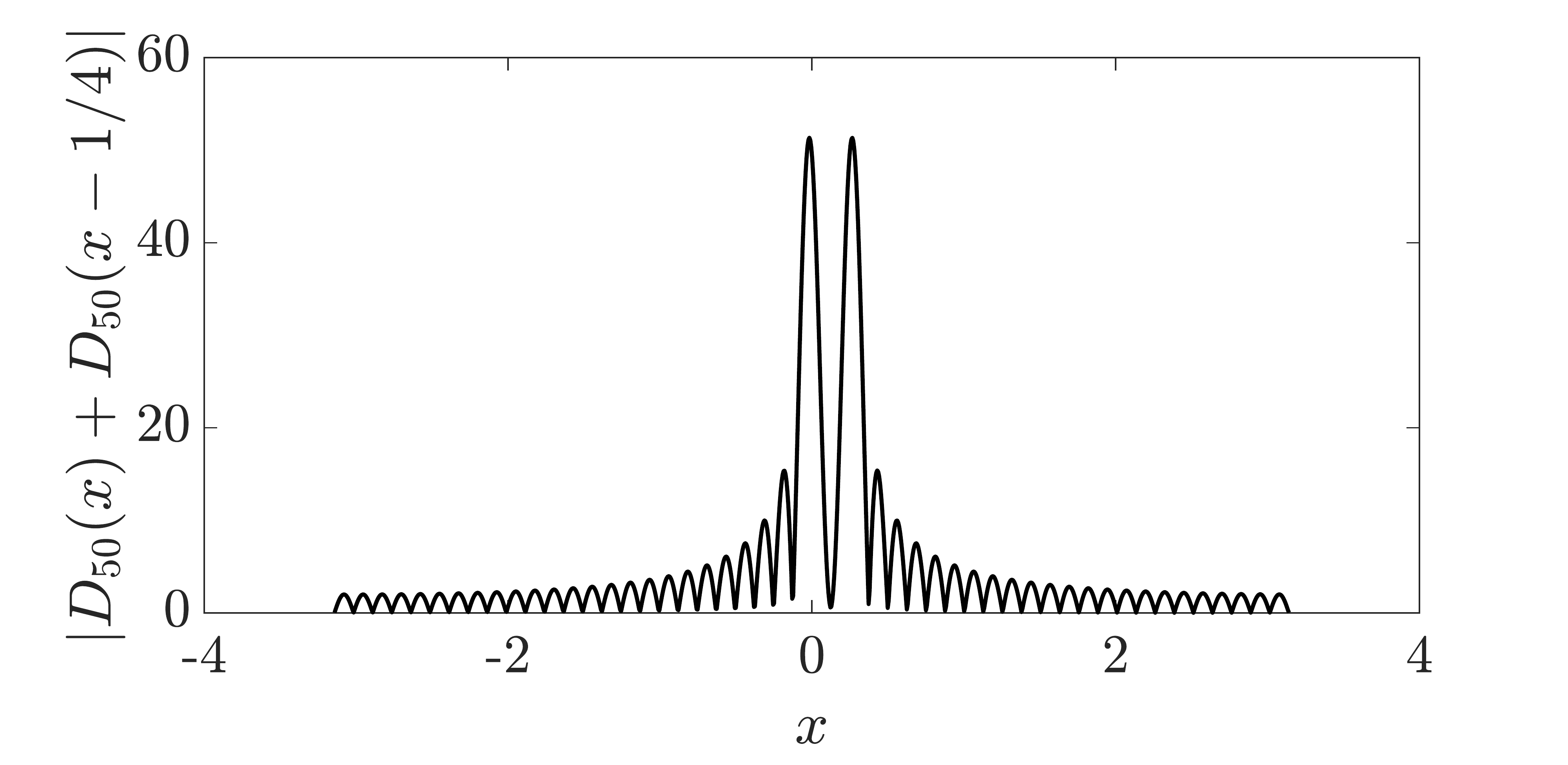}}\\
\subfloat[$\alpha = 1/8$]{\includegraphics[height=1in]{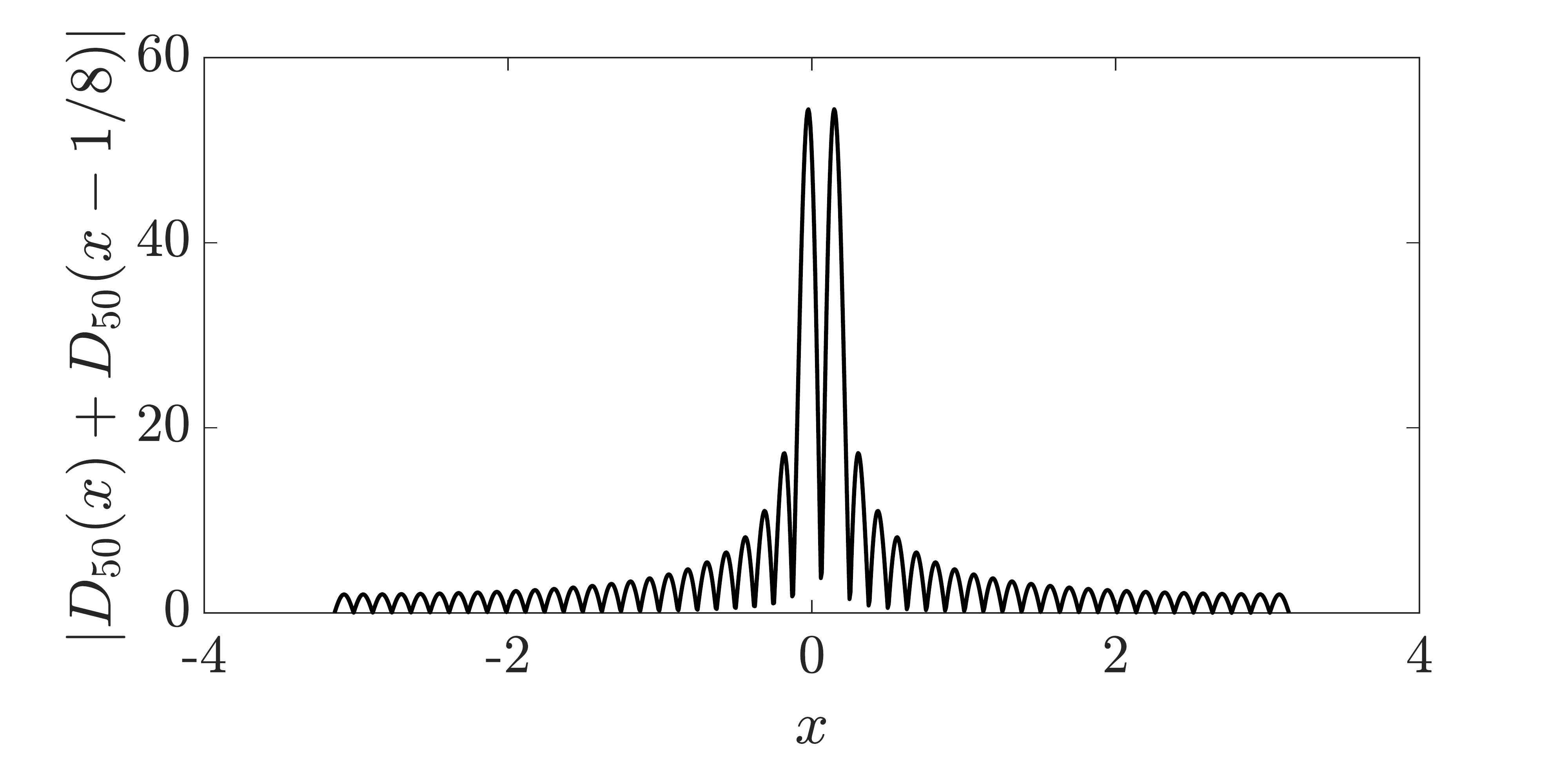}}
\subfloat[$\alpha = 1/16$]{\includegraphics[height=1in]{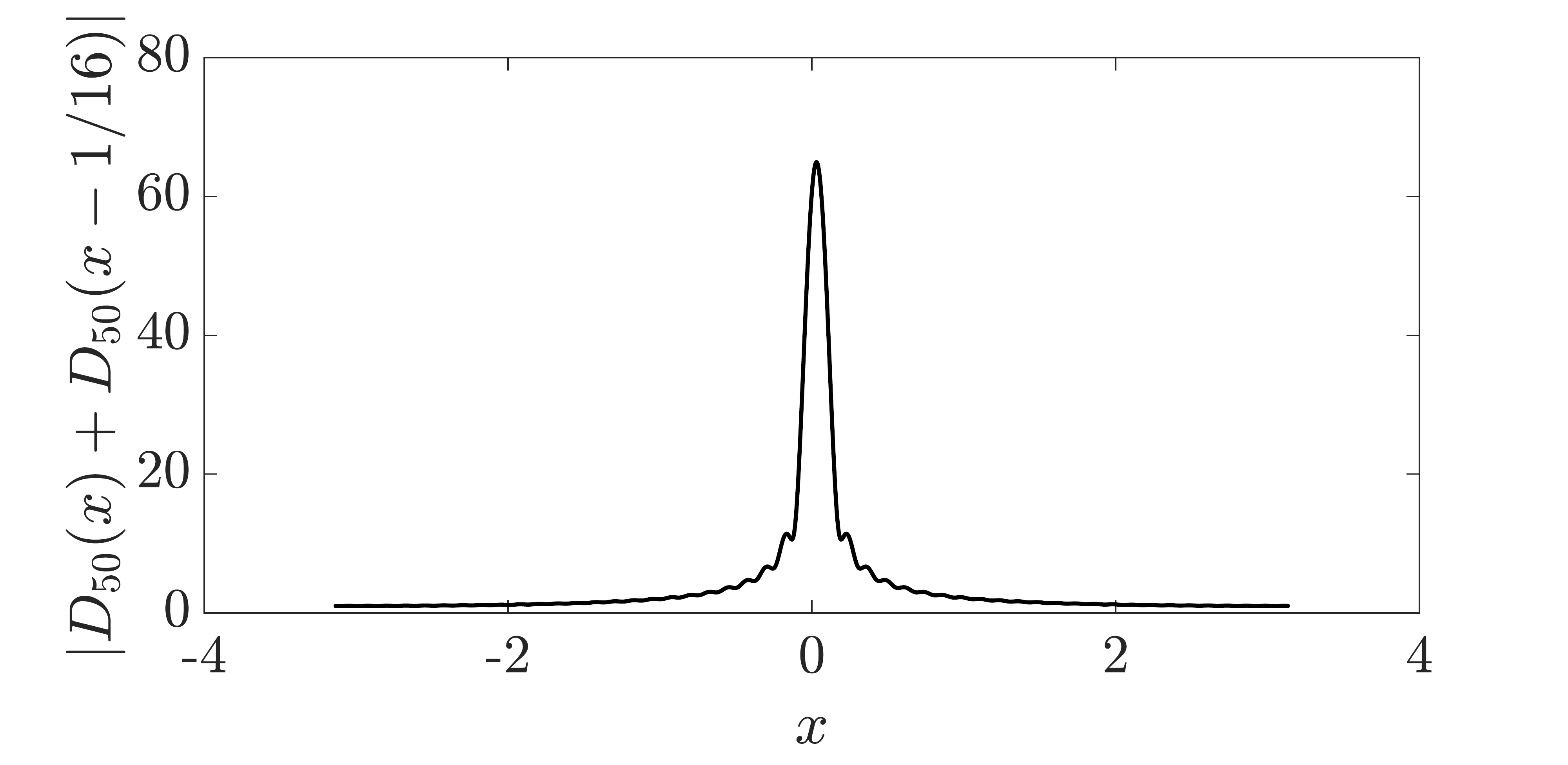}}
\caption{Plots of $|D_{50}(x)+D_{50}(x-\alpha)|$ for differing $\alpha$ in decreasing order from top-left to bottom-right.}\label{fig:DirichletKernel}
\end{figure}

\section{Sparse-DSM for the inverse medium problem}\label{sec:inverseScattering}
We continue now with our second application, inverse scattering of penetrable media from acoustic waves. Indeed, following for example \cite{CakoniColton2014}, assume time-harmonic acoustic waves with wavenumber $k>0$ are traveling through a medium whose index of refraction $n\in L^\infty(\mathbb{R}^d)$ only varies over a region of compact support. Assume a constant unitary index of refraction ($n=1$) outside of this region. Denote the contrast of the media by $m=n-1$ and by $D=\text{supp}(n-1)$ the shape and location of hidden scatterers to be recovered. Under these assumptions, a scattering experiment proceeds when an incident wave $u^i(\cdot;y)$ (whose form is specified below) is emitted from a point $y\in\Gamma_i\subseteq\mathbb{R}^{d-1}$. The resulting scattered field $u^s(x;y)$ measured at $x\in\Gamma_m\subseteq\mathbb{R}^{d-1}$ satisfies 
\begin{align*}
& \Delta u^s + k^2 nu^s = -k^2mu^i \\
&  \lim_{r\to\infty} r^{-(d-1)/2} \left(\frac{\partial u^s}{\partial r} - i k u^s\right) = 0 \quad \text{ where } \,\, r=|x|.
\end{align*}
The (Sommerfeld) radiation condition is required for uniqueness. Assume the scattered field is measured on the same surface as the surface from which incident fields are emitted,  $\Gamma_i=\Gamma_m=:\Gamma$ where $\Gamma$ is a closed, connected, and bounded surface containing $D$
 
The qualitative methods we discuss in this article depend on a near-field data operator $\mathcal{N}: L^2(\Gamma)\to L^2(\Gamma)$ defined by 
\begin{equation*}
(\mathcal{N}g)(x)= \int_{\Gamma} u^s(x ;y) g(y) \, \dx{S(y)}, \quad \textrm{ for } \,\, x \in \Gamma.
\end{equation*}
The inverse scattering problem is to \textit{estimate $D$ from $u^s(x;y)$, $x,y\in\Gamma$.}

The factorization of $\mathcal{N}$ is derived using the Lippmann-Schwinger equation \cite{ColtonKress2013}
\begin{equation}\label{eqn:LS}
u^s(x;y) = -k^2\int_{D} m(z) \Phi_k(x,z)\left(u^s(x;y)+u^i(x;y)\right)\dx{V(z)},
\end{equation}
where  $\Phi_k$ is the fundamental solution to the Helmholtz equation, 
 \begin{equation*}
  \Phi_k(x,y) := \begin{cases}
                             \frac{i}{4} H_0^{(1)}(k|x-y|) & d=2 \\
                             \frac{\exp{ik|x-y|}}{4\pi|x-y|} & d=3. 
                            \end{cases}
 \end{equation*}
Following for example \cite{AudibertHaddar2017, HarrisRome2017}, $\mathcal{N} = \mathcal{G}\mathcal{A}\mathcal{H}$ where $\mathcal{H}:L^2(\Gamma)\to L^2(D)$ is defined by
\begin{equation*} 
(\mathcal{H}g)(z) = \int_\Gamma u^i(z;y)g(y)\dx{s(y)}, \quad \textrm{ for } \,\, z\in D
\end{equation*}
and $\mathcal{G}:L^2(D)\to L^2(\Gamma)$ is defined by
\begin{equation*}
(\mathcal{G} f)(x) = \int_{D} \Phi_k(x,z)f(z)\dx{V(z)}, \quad \textrm{ for } \,\, x\in\Gamma. 
\end{equation*} 
The operator $\mathcal{A}:L^2(D)\to L^2(D)$ is defined as
\begin{equation*}
(\mathcal{A}\psi)(z) = -k^2m(z)\left(w(z)+\psi(z)\right), \quad  \textrm{ for } \,\, x\in D
\end{equation*}
where $w\in H^1_{\text{loc}}(\mathbb{R}^d)$ the solution to 
\begin{align}\label{eqn:helmholtz}
&\Delta w + k^2 n w = -k^2m \psi \qquad \text{ in } \mathbb{R}^d \\
& \lim_{r\to\infty} r^{-(d-1)/2} \left(\frac{\partial w}{\partial r} - i k w\right) = 0 \quad \text{ where } \,\, r=|x|.
\end{align}

\subsection{Infimum criterion for near-field scattering}
Three hypothesis need to be met in order to relate $D$ to the solution of an optimization problem, according to Theorem \ref{thm:coerciveInfCrit} and Corollary \ref{cor:extendedEstimationInf}. The first is to demonstrate $\mathcal{N}$ has a symmetric factorization, which is not immediate from the factorization given above. The second is that the middle operator of this factorization is coercive. Finally, we must relate $D$ to $\mathcal{R}(\mathcal{H}^*)$ in accordance with Corollary \ref{cor:extendedEstimationInf}. Each of these points has been the subject of recent research, which we summarize and discuss here. 

To meet the requirement that $\mathcal{N}$ have a symmetric factorization, $\mathcal{N}=\mathcal{H}^*\mathcal{A}\mathcal{H}$, we need to show $\mathcal{G}=\mathcal{H}^*$. Formally, a short calculation reveals that
\begin{equation*}
\left(\mathcal{H}^*f\right)(x) = \int_{D} \overline{u^i(z;x)}f(z) \dx{V(z)}, \quad x\in\Gamma. 
\end{equation*}
This suggests using non-physical incident waves of the form $u^i(x;y)=\overline{\Phi_k(x,y)}$ leading to $\mathcal{G}=\mathcal{H}^*$. This is a common practice \cite{Arens2011, ArensKirsch2003, Kirsch2007, HarrisRome2017} and it can be shown that these non-physical incident fields can be approximated arbitrarily-well with physical incident fields\footnote{Another option is to use physical incident waves $u^i(x;z)=\Phi_k(x,y)$, but process the data differently. Indeed, \cite{AudibertHaddar2017, HuEtAl2014} analyze an operator $\mathcal{L}=\mathcal{B}\mathcal{N}$ where $\mathcal{B}$ is some linear operator which is both computable and which modifies $\mathcal{N}$ such that $\mathcal{L}$ has the correct factorization.}.

The second difficulty in applying Theorem \ref{thm:coerciveInfCrit} is the need for $\mathcal{A}$ to be coercive. There are numerous assumptions on $m$ known to ensure this \cite{AudibertHaddar2017,CakoniColtonHaddar2016}. We make the following assumption throughout that is sufficient for coercivity of $\mathcal{A}$. 

\begin{assumption}[Adapted from \cite{AudibertHaddar2017,HarrisRome2017}]\label{ass:coercive}
Assume that $n\in L^\infty(D)$ and that there is a constant $n_0>0$ such that either 
\begin{itemize} 
\item $\text{Im }n(x) \geq n_0$ for almost all $x\in \overline{D}$ 
 \item $1-n(x)\geq n_0$ for almost all $x\in \overline{D}$ 
 \item $n(x)-1\geq n_0$ for almost all $x\in \overline{D}$.
 \end{itemize}
\end{assumption}

Under this Assumption \ref{ass:coercive}, we have that there exists a $c>0$ such that for any $\varphi\in\mathcal{R}(\mathcal{H})\subset L^2(D)$, 
$c\|\varphi\|^2_{L^2(D)} \leq \big| (\mathcal{A}\varphi , \varphi)_{L^2(D)} \big|$ and hence can apply Theorem \ref{thm:coerciveInfCrit}.

The final assumption required to apply an infimum criterion for the localization of $D$ is relating $D$ to $\mathcal{R}(\mathcal{H}^*)$. This relies on statements about an interior transmission eigenvalue problem,
\begin{equation}\label{eqn:TEv}
\begin{split}
&\Delta u + k^2n u = 0 \quad \text{ in } D,  \qquad\qquad
\Delta v + k^2v = 0 \quad \text{ in } D,  \\
&(u-v) = 0 \quad \text{ on } \partial D, \qquad\qquad
\pd{}{\nu}(u-v) = 0 \quad \text{ on } \partial D
\end{split}
\end{equation}
The corresponding \textit{transmission eigenvalues} are the $k \in \mathbb{C}$ such that there are non-trivial solutions $(u,v)\in L^2(D)\times L^2(D)$ to \eqref{eqn:TEv} with difference $u-v\in H^2(D)$. The results we state below depend on the wavenumber of the incident field \textit{not} being a transmission eigenvalue. 

\begin{assumption}\label{ass:TEv}
The wavenumber $k$ is not a transmission eigenvalue for \eqref{eqn:TEv}. 
\end{assumption}
There are numerous results demonstrating that the set of transmission eigenvalues is almost empty (i.e., countable without finite accumulation points) \cite{CakoniColtonHaddar2016}. For example, if $m$ is strictly positive in a neighborhood around $\partial D$, then the set of transmission eigenvalues is empty almost everywhere. This assumption provides a link between the location of $D$ and an infimum criterion. 

\begin{lemma}[\cite{HarrisRome2017}, Theorem 3.2]
Let $\mathcal{Z}\subseteq\mathbb{R}^d$ be a sampling region such that $D\subseteq\mathcal{Z}\subseteq \text{interior}(\Gamma)$. Under Assumption \ref{ass:TEv},
\begin{equation*}
\Phi_k(\cdot,z)\in\mathcal{R}(\mathcal{H}^*) \qquad \text{ if and only if } \qquad z\in D. 
\end{equation*}
\end{lemma}

All together, we thus have an infimum criterion to serve as the basis for a sparse-DSM for inverse acoustic scattering. 

\begin{corollary}\label{cor:infCritScattering}
Under Assumptions \ref{ass:coercive} and \ref{ass:TEv}, 
\begin{equation*}
z\in D \quad \text{ if and only if } \quad \mathcal{I}^{(1)}(z) = \inf_{g\in X}\left\{\left|\left(g,\mathcal{N}^*g\right)_{L^2(\Gamma)} \right|^2\,:\, (g,\Phi_k(\cdot,z))_X=1\right\}>0. 
\end{equation*}
\end{corollary}

\subsection{Direct sampling method for extended obstacles}
To derive a DSM, we couple Corollary \ref{cor:infCritScattering} with the discussion in Section \ref{sec:beamforming}. Indeed, if $z\in D$, $\mathcal{I}_{\emptyset}^{(1)}(z):=\left|\left(\mathcal{N}\Phi_k(\cdot,z),\Phi_k(\cdot,z)\right)_{L^2(\Gamma)}\right|^2>0$. Using this $\mathcal{I}_{\emptyset}^{(1)}(z)$ as a basis for a DSM is similar to those discussed in \cite{LeemLiuPelekanos2018,Liu2017}.

To understand the behavior of $\mathcal{I}_{\emptyset}^{(1)}(z)$ away from $D$, assume that $\Gamma=\partial B_R$ where $B_R$ is the ball centered at the origin with radius $R$. It is known \cite{ChenChenHuang2013} that
\begin{equation*}
(\mathcal{H}\Phi_k(\cdot,z))(x) = \int_{\partial B_R} \overline{\Phi_k(x, y)} \Phi_k(z, y) \, \dx{S(y)}  = \text{Im }{\Phi_k(z, x)} + w(z, x)
\end{equation*}
where $w$ satisfies the estimate $\| w\|_{W^{1,\infty}(\mathcal{Z} \times \mathcal{Z})} \leq CR^{1-d}$ so long as $D \subset \mathcal{Z} \subset B_R$. The factorization of $\mathcal{N}$ and the definition of $\Phi_k$ implies that there is a $C>0$ so that
\begin{align*}
\mathcal{I}_{\emptyset}^{(1)}(z)= \left|\left(\mathcal{A} \mathcal{H}\Phi_k(\cdot,z), \mathcal{H} \Phi_k(\cdot,z)\right)_{L^2(D)}\right|^2  &\leq C \big\| \mathcal{H} \Phi_k(\cdot,z)  \big\|^4_{L^2(D)}\\
                                                                                                               &= C \big\| \text{Im }{\Phi_k ( z , x)} + w(z\, , x) \big\|^4_{L^2(D)}\\
                                                                                                               &\leq C\left( \mathrm{dist}(z,D)^{\frac{1-d}{2} } + R^{1-d} \right)^4.
\end{align*}
Here we have used the property that the imaginary part of the fundamental solution, $J_0 (t)$, which has a decay rate of $t^{-1/2}$ as $t \to \infty$ for the $\mathbb{R}^2$ case and for the case in $\mathbb{R}^3$, we have used that $j_0 (t)$ has  decay rate of  $t^{-1}$ as $t \to \infty$. 
This implies that the indicator $\mathcal{I}_0^{(1)}(z)$ decays as the sampling point $z$ moves away from the scatterer $D$ at a rate related to the distance between $z$ and $D$ and the distance between $D$ and measurement locations. 

\subsection{Direct sampling for point obstacles}
In the case of weakly-scattering point scatterers, we can adapt our DSM to use Theorem \ref{thm:boundedInfCrit} instead of Theorem \ref{thm:coerciveInfCrit}. In doing this, we derive a DSM more similar to that introduced in \cite{Ito2012, Ito2013, Potthast2010} and elsewhere. Indeed, assume $D=\cup_{j=1}^J D_j$ is a union of small balls centered at points $z_j\in\mathbb{R}^d$ each with radius $\epsilon\ll 1$. Let $m(z)=\sum_{j=1}^J m_j\mathds{1}_{D_j}(z)$ be constant on each ball with $m=\mathcal{O}(1)$. Neglecting multiple scattering, using $u^i(x;y)=\overline{\Phi_k(x,y)}$ for consistency, and using the Lippmann-Schwinger equation \eqref{eqn:LS} yields the point-obstacle Born approximation \cite{ColtonKress2013}
\begin{equation*}
u^s_B(x;y) := -k^2\sum_{j=1}^J m_j\Phi_k(x,z_j)\overline{\Phi_k(y,z_j)}.
\end{equation*}
In this case, the near field operator becomes
\begin{equation*}
(\mathcal{N}_b g)(x) = -k^2\sum_{j=1}^J m_j\Phi_k(x,z_j)\int_{\Gamma}\overline{\Phi_k(y,z_j)}g(y)\dx{s(y)}. 
\end{equation*}
Using the linear independence  \cite{CakoniRezac2017} of the sets $\{\Phi_k(x,z_j), x\in\Gamma, j=1,\ldots,J\}$ and $\{\overline{\Phi_k(y,z_j)}, y\in\Gamma, j=1,\ldots,J\}$ allows us to apply Corollary \ref{cor:pointEstimationInf} to the Born approximation data and leads to the DSM indicator functional
\begin{equation}\label{eqn:bornDSM}
\mathcal{I}_{\emptyset}^{(0)}(z) = \|\mathcal{N}_b\Phi_k(\cdot,z)\|_{L^2(\Gamma)}^2. 
\end{equation}
In the same way as for extended inhomogeneities, 
\begin{equation*}
\mathcal{I}_{\emptyset}^{(0)}(z)\leq C \max_j \left(\text{Im}\Phi_k(z,z_j)+R^{1-d}\right)^2 \leq  C  \max_j \left( \mathrm{dist}(z,z_j)^{\frac{1-d}{2} } + R^{1-d} \right)^2
\end{equation*}
for some $C>0$. As $z\to z_j$, we expect to see peaks in $\mathcal{I}_{\emptyset}^{(0)}(z)$ and as $z$ moves away from each $z_j$ we expect to see valleys. 

Interestingly, the same bound holds true (with a different constant) for $\mathcal{I}_{\emptyset}^{(0)}(z)$ when the scattered field is only measured at one location. This DSM, with indicator functional $\mathcal{I}_{\emptyset}^{(0)}(z)=\left|\mathcal{N}_b\Phi_k(x,z)\right|^2$, has same form as those in \cite{Ito2012, Ito2013, Potthast2010}. Requiring less data is a big benefit of the DSMs over classical qualitative methods. This form of the DSM, however, loses its relationship with \ref{thm:boundedInfCrit}.

\section{Computational Results}\label{sec:numerics}
In this section we will present a number of numerical examples demonstrating the effectiveness of our new scheme, as well as how it relates to more commonly-used qualitative methods. All simulated examples will be polluted with additive Gaussian noise of the form $f^{\text{measured}} = f^{\text{clean}} + \sigma \epsilon. $ Here, $\epsilon$ is a Gaussian random variable with mean $0$ and variance $1$ in its real and imaginary parts. The standard deviation $\sigma$ is chosen to ensure the measured data has a particular signal-to-noise ratio (SNR); if there are $N$ measurements, we define $\sigma$ according to $\sigma = \|f\|/\sqrt{N\times SNR} $, which implies that SNR $\approx \|f\|/\mathbb{E}\|\sigma \epsilon\|$. We remark that for the purposes of the examples we consider below, the \textit{values} of th	e indicator functions are meaningless. They can be normalized or transformed with a monotonic function to any value with no change in interpretation and hence should not be compared between different algorithms. Indeed, in Section \ref{sec:timeHarmonicPoint} we plot indicator functions on a $10\log_{10}$ scale which better captures their full range.

The first example is a particularly simple case with closed-form expressions for data, which we use to show some basic properties of sparse-DSM. The simulated data for the inverse scattering examples were calculated with a volume integral equation method described in \cite{Rezac2017}. It uses a low-order Galerkin discretization of the Lippmann-Schwinger equation \ref{eqn:LS}. 

We also apply the sparse-DSM to experimentally-measured data in \ref{sec:measuredAOA}. The measurement, which is described in more detail in that section and in \cite{WeissEtAl2019}, allows for AOA estimation in a realistic situation. It was collected by staff in the Communications Technology Laboratory at the National Institute of Standards and Technology.

\subsection{A 1-Dimensional AOA Example}\label{sec:simulatedAOA}
We first analyze a AOA estimation problem in one dimension. The data in this simplified problem takes the form
\begin{equation}\label{eqn:1DAOA}
f(x_m,t_{\ell}) = \sum_{j=1}^J \alpha_j s_j(t_{\ell}) \exp{(-ik m\Delta x u_j)}, \quad m=1,\ldots,M \text{ and } \ell=1,\ldots,L
\end{equation}
where $k>0$ and $\Delta x>0$ are considered known and $u_j\in[-1,1]$ are the unknowns-of-interest. For simplicity, we consider $s_j(t_\ell) = \exp{(-i \ell\Delta t u_j)}$ which are linearly independent for different $j$. This problem is very similar to the sparse spectral density estimation problem of finding the best representation of a signal in terms of a sparse Fourier series for which there are many applications and solution techniques \cite{Hokanson2013}. Since $s_j$ and $\varphi(x;u_j):=\exp{(-ik x u_j)}$ are linearly independent for different $j$, we can estimate $u_j$ based on Corollary \ref{cor:pointEstimationInf} . 

In Figure \ref{fig:1AOA_8points}, we show the transitioning of $k$-DSM from a standard DSM to a full infimum-criterion. There are 8 unknowns randomly placed in $[-1,1]$ in this problem indicated by semi-transparent vertical lines. Data is collected on $x_m=t_m=m\Delta$ for $m=0,\ldots,19$ and $\Delta=50/19$. Very small amounts of noise (SNR$\approx 3\times 10^5$) are added to the measurements. The figure indicates that the DSM cannot distinguish between each of the unknowns, particularly the middle ones. The infimum-criterion works well here because there is little noise and a larger amount of data. We also see that as $k_{\text{DSM}}$ increases, the estimates produced by the sparse-DSM improve; when $k_{\text{DSM}}=4$, estimates from the sparse-DSM are sharper than the DSM, but the right-most unknowns are still blurred. When $k_{\text{DSM}}=8$, the sparse-DSM performs as well as the inf-criterion. 

\begin{figure}[tbhp]
    \centering
	\subfloat[DSM]{\includegraphics[height=1.5in]{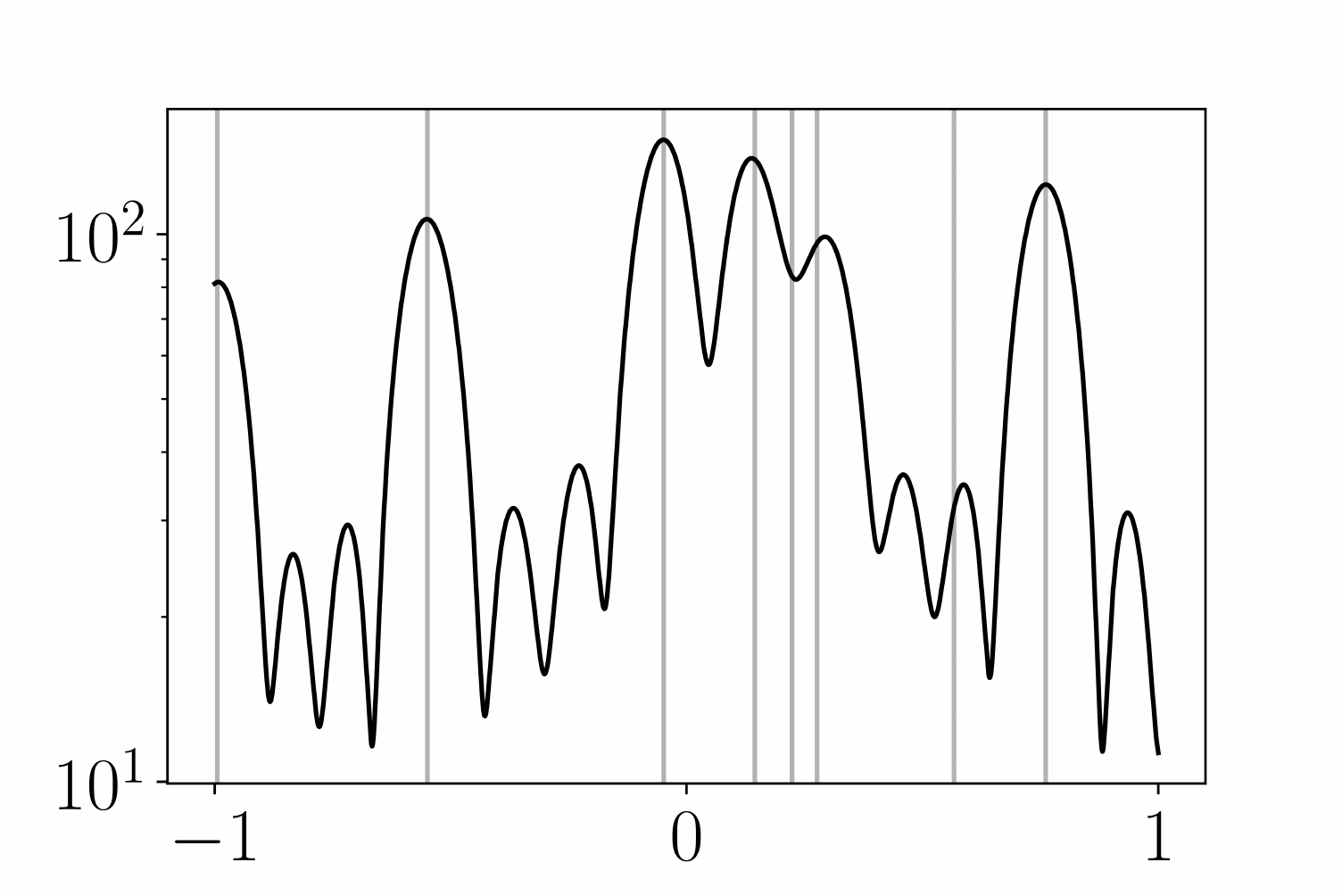}}
        \subfloat[$k$-DSM, $k_{\text{DSM}}=4$]{\includegraphics[height=1.5in]{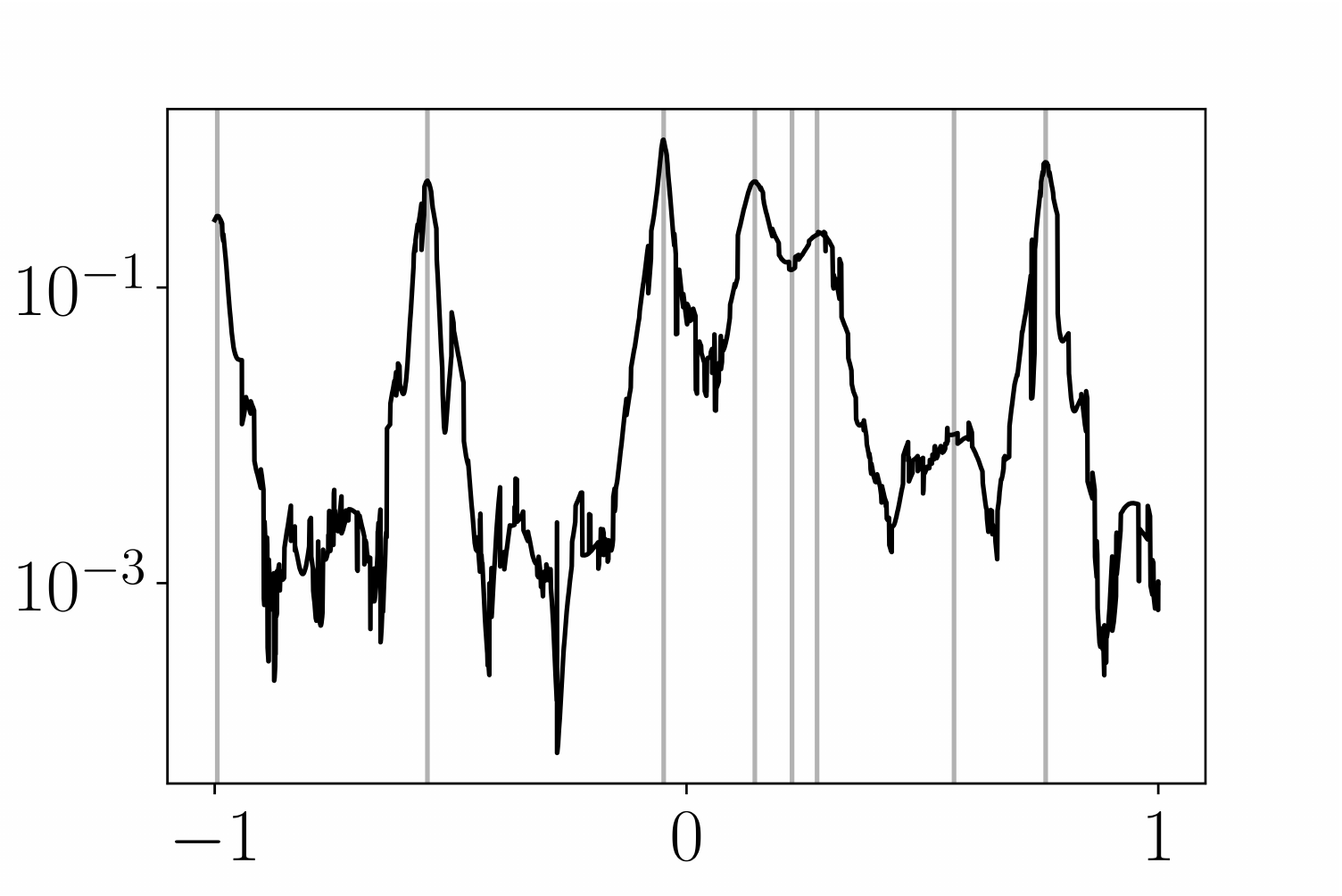}}\\
        \subfloat[$k$-DSM, $k_{\text{DSM}}=8$]{\includegraphics[height=1.5in]{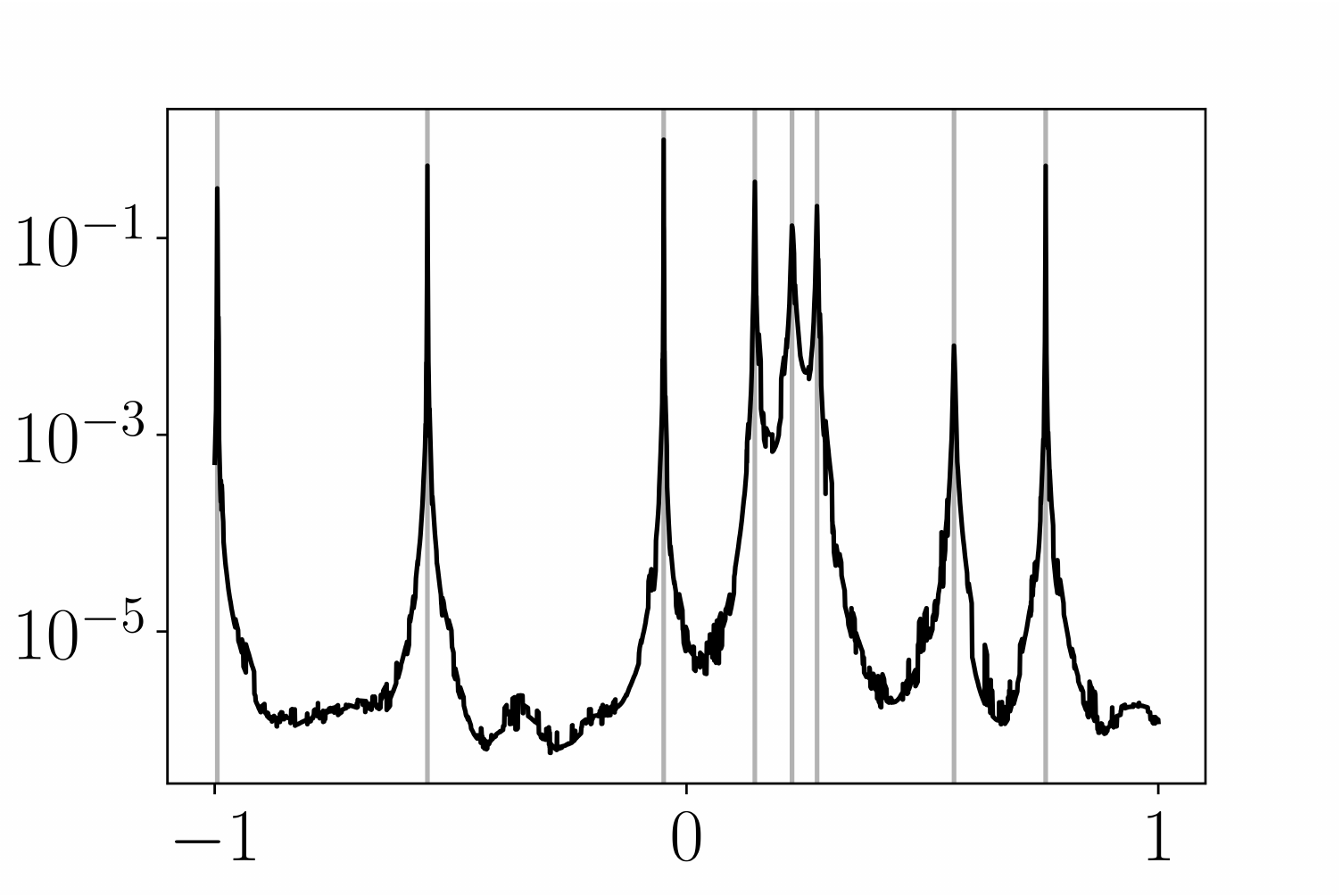}}
        \subfloat[Infimum-criterion estimate]{\includegraphics[height=1.5in]{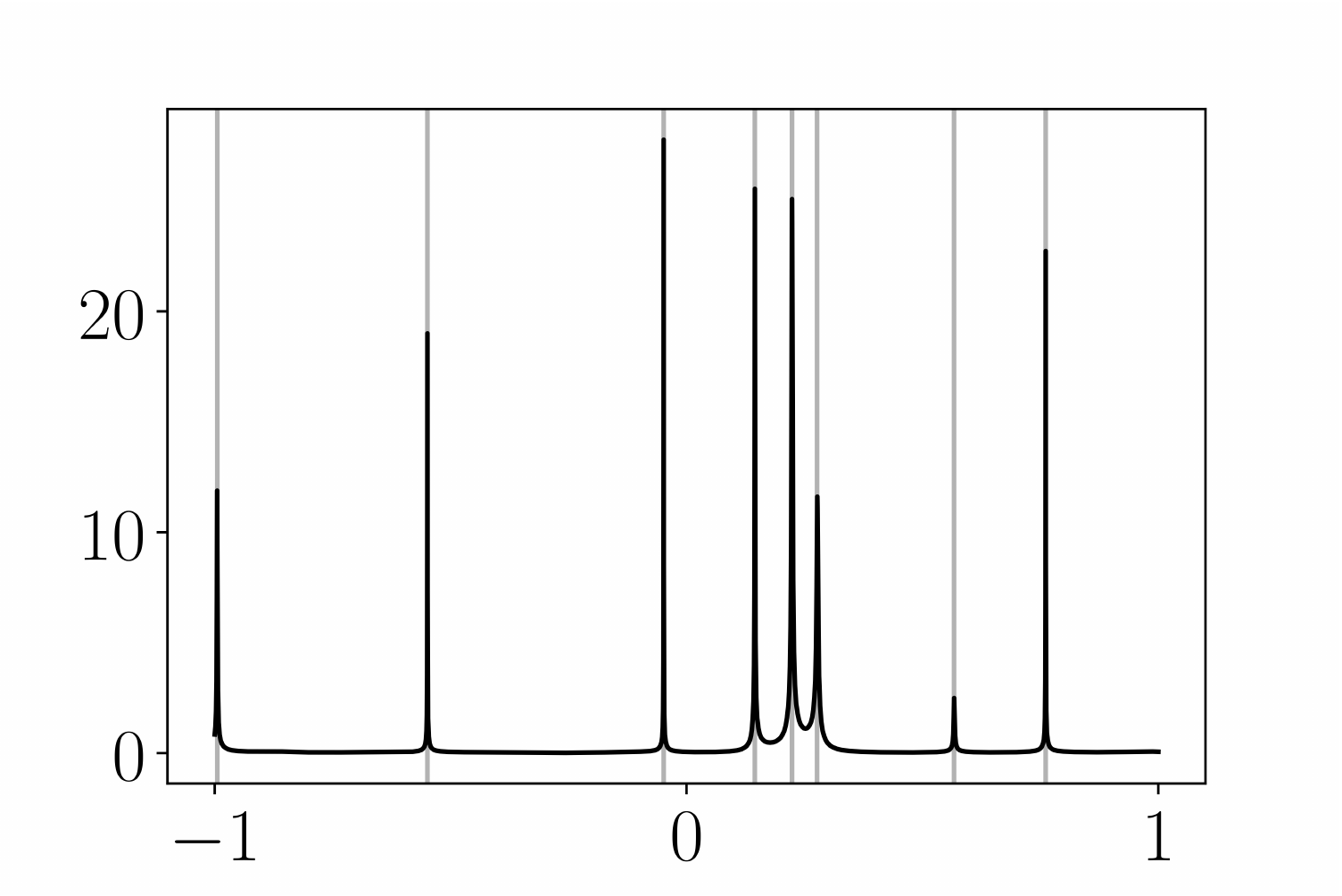}}
    \caption{Reconstructions of 8 1-dimensional AOAs.}\label{fig:1AOA_8points}
\end{figure}

One benefit of qualitative methods over traditional iterative-based techniques is their speed. This is particularly true of DSMs. Table \ref{tab:timing} compares an estimate of computational time between the DSM, $k$-DSM, and inf-criterion in the example discussed above. We calculate these times for two different data matrix sizes: $20\times 20$ and $100\times 100$. As expected, the DSM is significantly-faster than the other techniques - although it also does not perform as well. On the other hand, $k$-DSM is slower than the infimum-criterion on the small data example, but significantly faster for the larger-sized data matrix. While differences in implementation and computing resources effect on this sort of analysis (and so we omit more details), we expect that DSMs will typically be much faster than sparse-DSMs or infimum-criteria while sparse-DSMs will be of similar speed as infimum-criteria on small amounts of data and significantly faster for larger amounts of data. These computational trade-offs are similar for the each of more realistic examples which follow.

\begin{table}[htb]
\centering
\begin{tabular}{c||c|c}
  & $20\times 20$ measurements (s) & $100\times 100$ measurements (s) \\
\hline
DSM & 0.48 & 1.8 \\
$k$-DSM ($k_\text{DSM}=4$) & 3.7 & 6.1 \\
$k$-DSM ($k_\text{DSM}=8$) & 6.7 & 9.1 \\
Infimum-criterion & 3.1 & 56 
\end{tabular}
\caption{Timing of qualitative methods applied to the problem in Figure \ref{fig:1AOA_8points}.}\label{tab:timing}
\end{table}

\subsection{Reconstructions for AOA estimation with measured data}\label{sec:measuredAOA}
We now analyze data from a AOA estimation measurement done at the Communications Technology Laboratory at the National Institute of Standards and Technology. The data we analyze is a result of incident waves which were transmitted into an area with two metal rods. The resulting scattered waves were measured at an array of receivers. Measurements were done in the frequency domain on a grid of 26.5GHz to 40GHz with a 10MHz frequency step and were converted into the time domain prior to analysis. The spatial measurement array was a uniformly-spaced two-dimensional grid with a spacing of approximately 3.7mm. The measurement data consisted of 1225 spatial measurement locations each with 1350 time steps. Full details on the experimental procedure are given in \cite{WeissEtAl2019}. 

True AOAs were not known for this dataset, so we compare our techniques to the results of the MUSIC algorithm. MUSIC is known to perform well on this type of data, though there can be difficulties when the true number of AOAs is unknown \cite{Schmidt1986}. Figure \ref{fig:samurai_close} shows the estimated AOA values using different estimation techniques. Each algorithm is run on the sampling grid $[-1,1]^2$ though Figure \ref{fig:samurai_close} is zoomed-in to $[-0.25,0.25]\times[-0.25,0.5]$. Beamforming does not distinguish between the cylinders as reliably as the other three methods, which perform well. Indeed, the beamforming results have extra peaks to the right and left of the true peaks and the two true peaks merge together more significantly than the other methods. This is related both to the resolution issues depicted in Figure \ref{fig:DirichletKernel} and to the related side-lobe phenomenon \cite{VanVeenBuckley1988}.

\begin{figure}[tbhp]
\centering
\subfloat[MUSIC]{\includegraphics[height=1.5in]{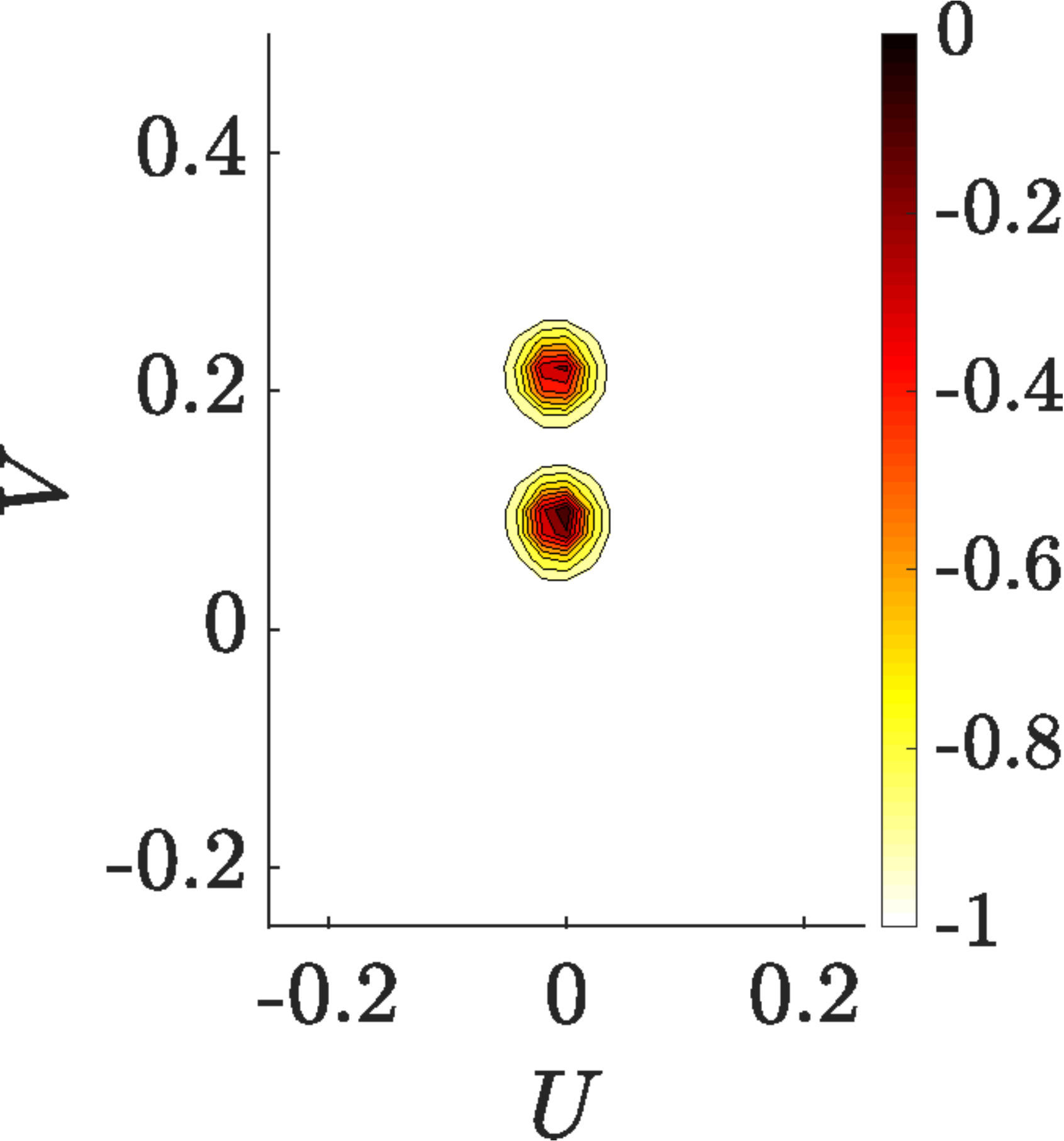}}\hfill
\subfloat[Beamforming]{\includegraphics[height=1.5in]{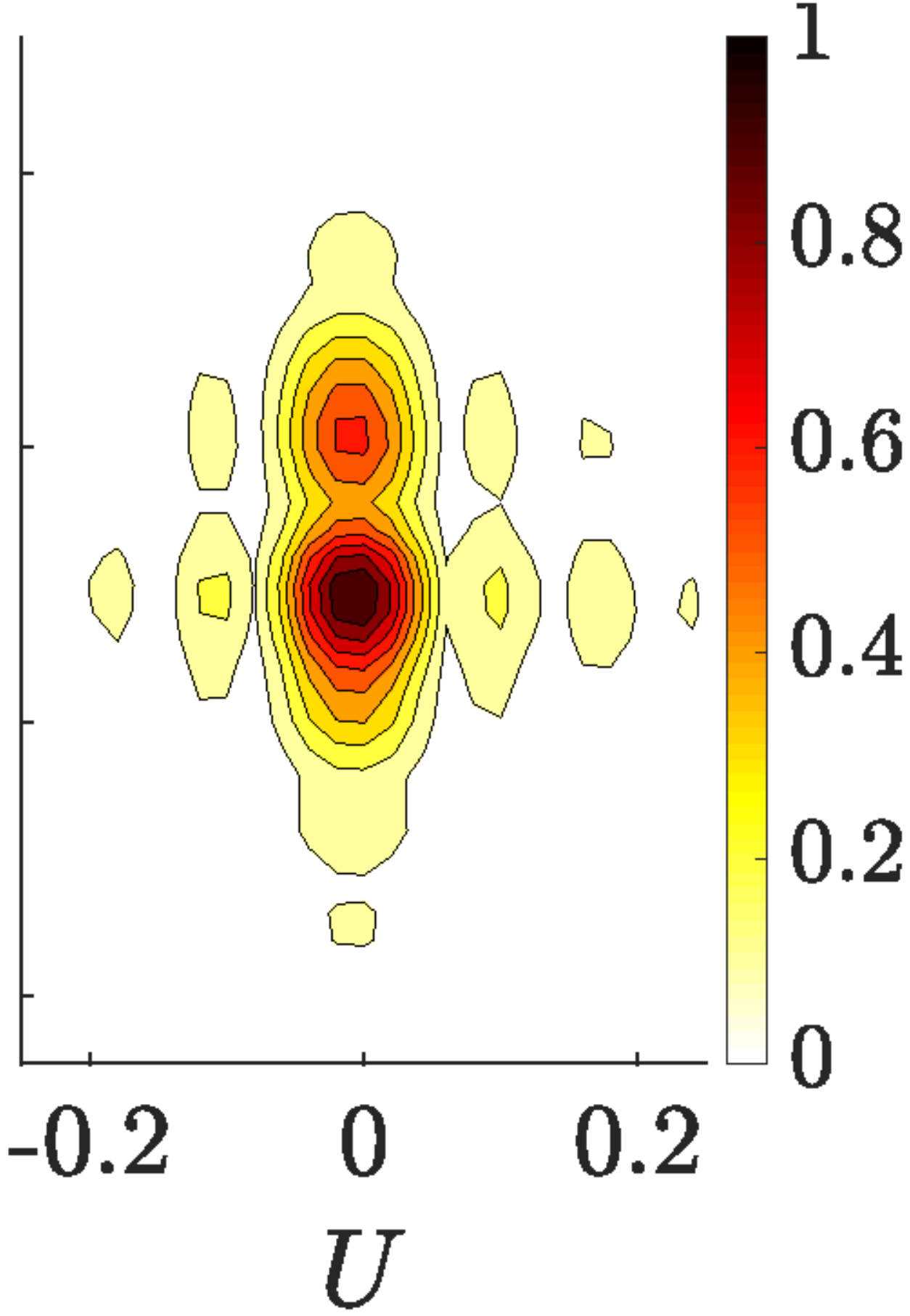}}\hfill
\subfloat[$k$-DSM, $k_{\text{DSM}}=1$]{\includegraphics[height=1.5in]{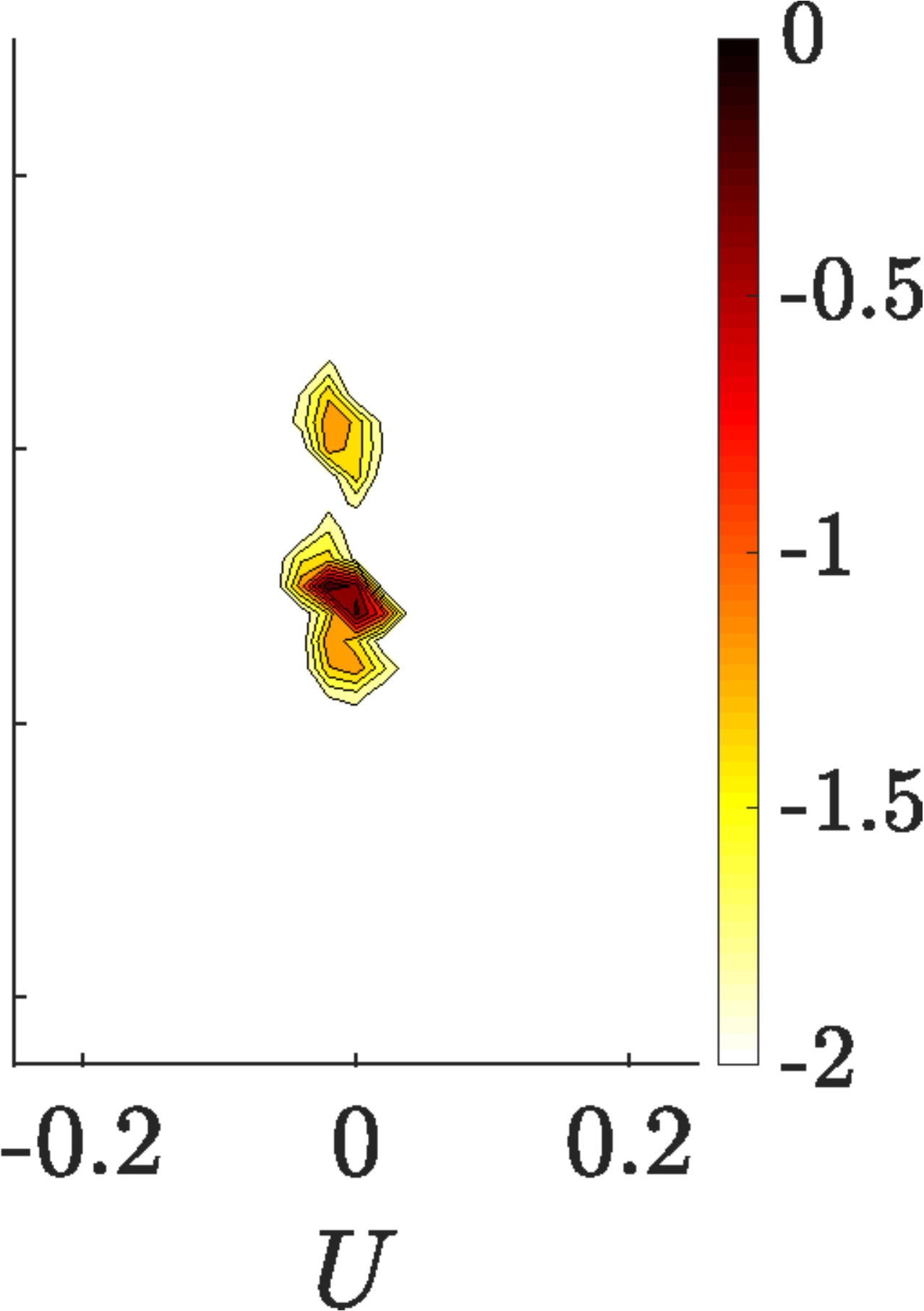}}\hfill
\subfloat[error-DSM, $\delta_{\text{DSM}}=0.94$]{\includegraphics[height=1.5in]{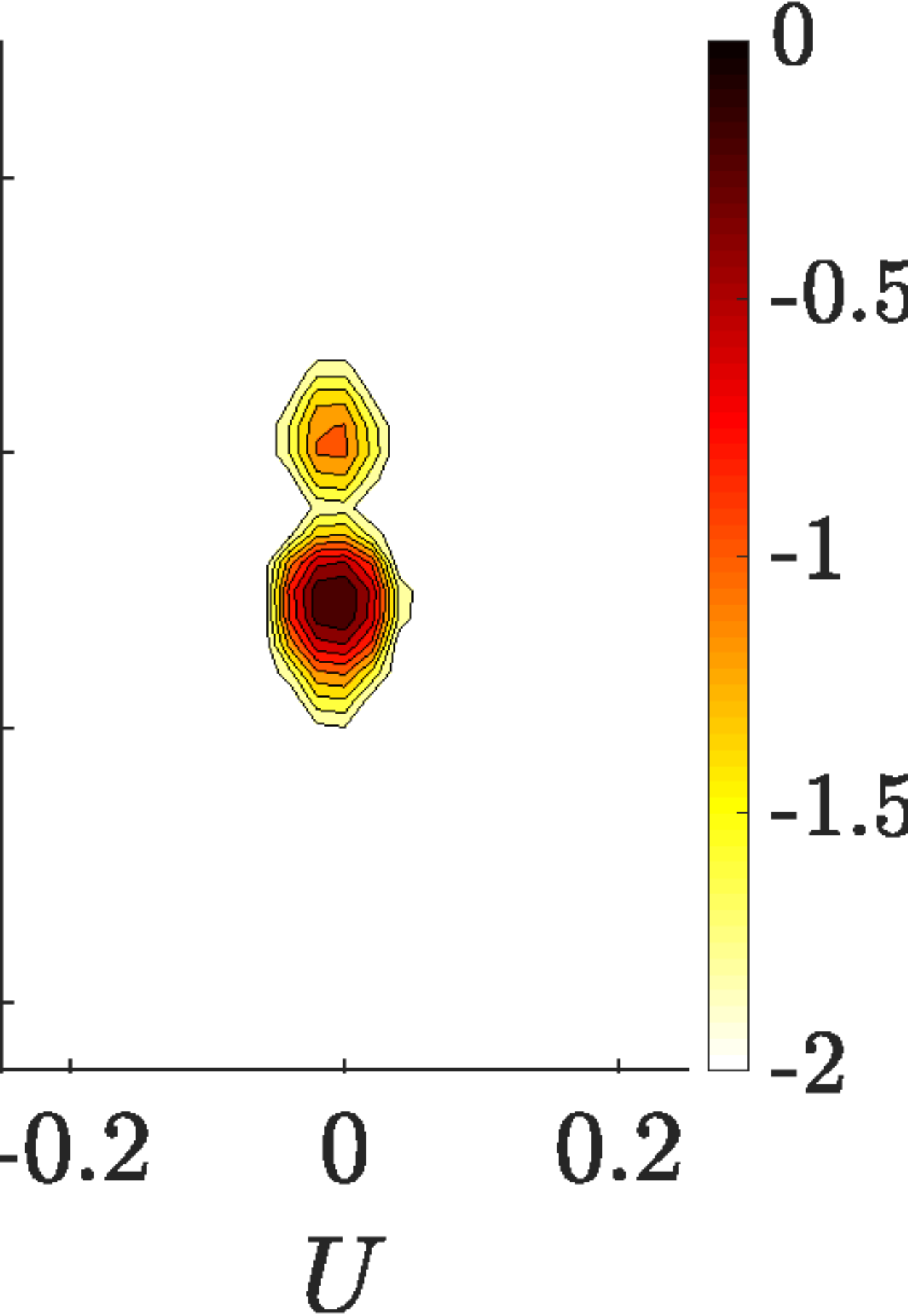}}
\caption{AOA estimates from two nearby cylinders.}\label{fig:samurai_close}
\end{figure}

\subsection{Reconstructions from acoustic scattering}\label{sec:timeHarmonicPoint}
Finally, we apply each of the sparse-DSMs to simulated acoustic scattering problems, as justified in Section \ref{sec:inverseScattering}. Unless otherwise noted, we use a wavenumber of $k=8$ and a circular multistatic measurement geometry centered at the origin with radius of 4 units.

The first example in Figure \ref{fig:th_3points} shows the reconstruction of three small circular unknowns whose true locations are $(0,1.5)$, $(1,0)$, and $(-1,-1)$. For this example, we take SNR=1000 indicating very low noise and measure on 25 tranceivers. The DSM and both sparse DSM algorithms perform well here, in particular the error-DSM - note the logarithmic scale of the plotted values. Since these are small scatterers, we used $\mathbf{I}^{(0)}(z)$ as an indicator functional here.  

\begin{figure}[tbhp]
\centering
\subfloat[DSM]{\includegraphics[height=1.5in]{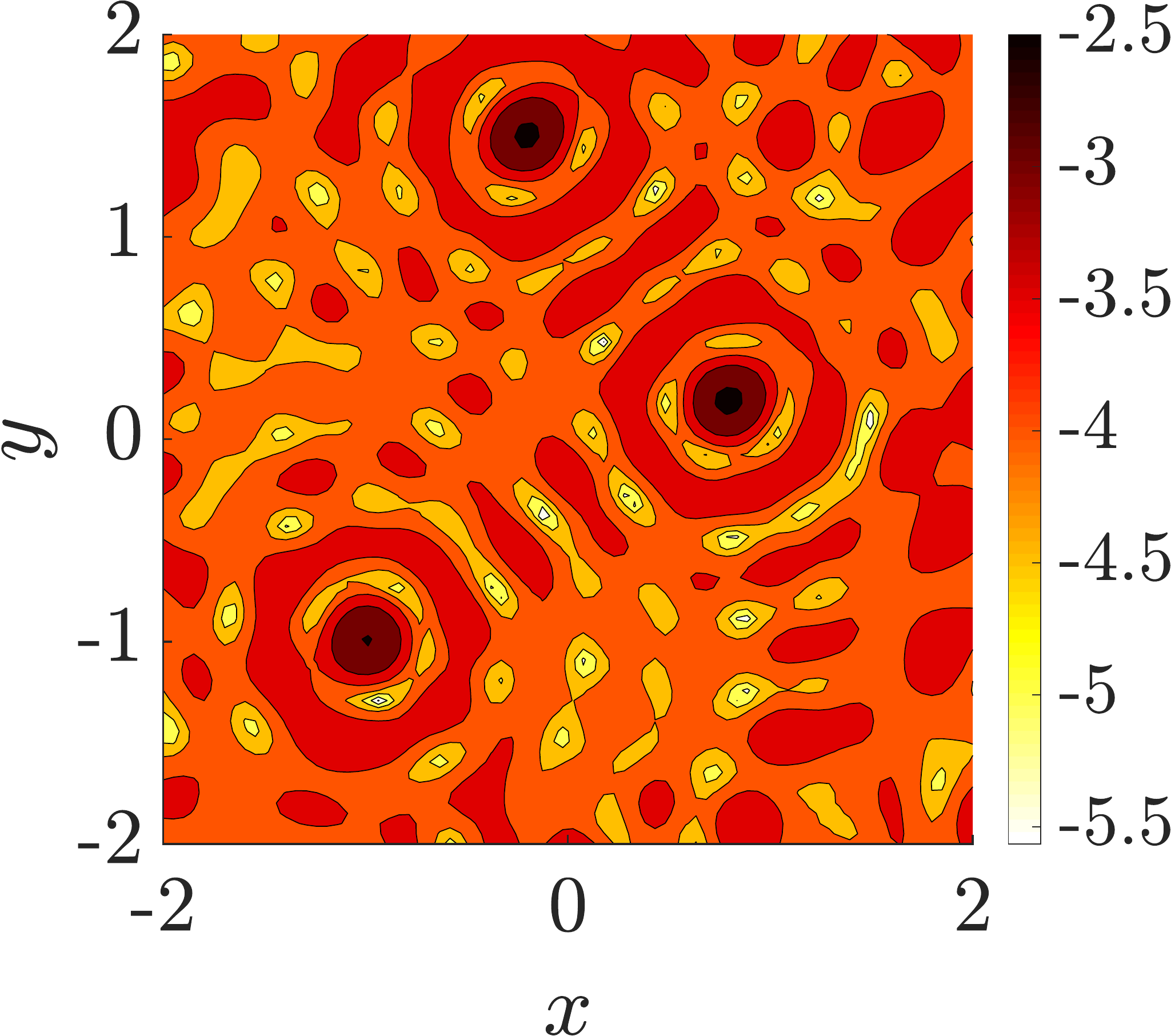}\label{fig:th_3points_DSM}}\hfill
\subfloat[$k$-DSM, $k_{\text{DSM}}=5$]{\includegraphics[height=1.5in]{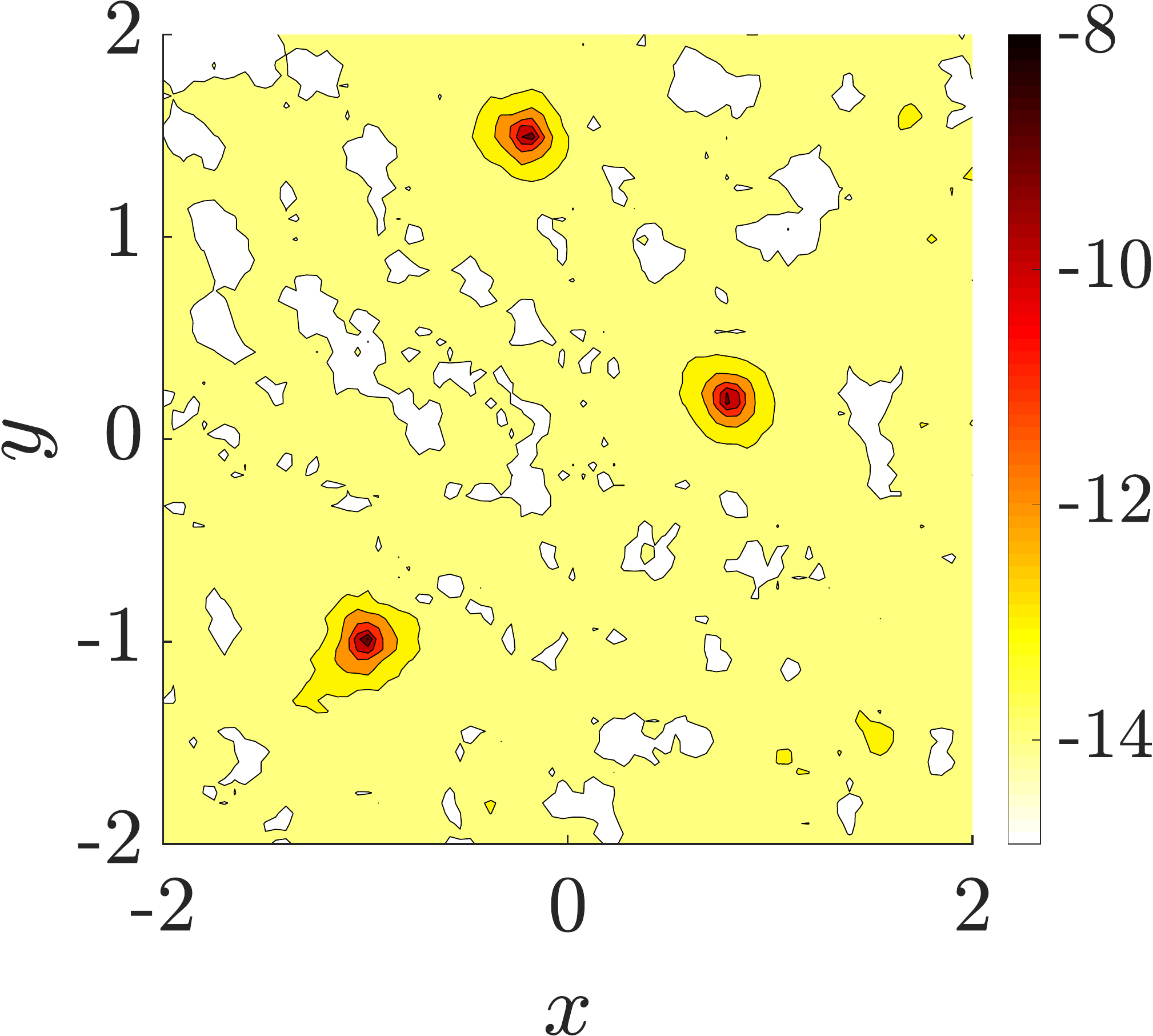}\label{fig:th_3points_kDSM}}\hfill
\subfloat[error-DSM, $\delta_{\text{DSM}}=5\times 10^{-7}$]{\includegraphics[height=1.5in]{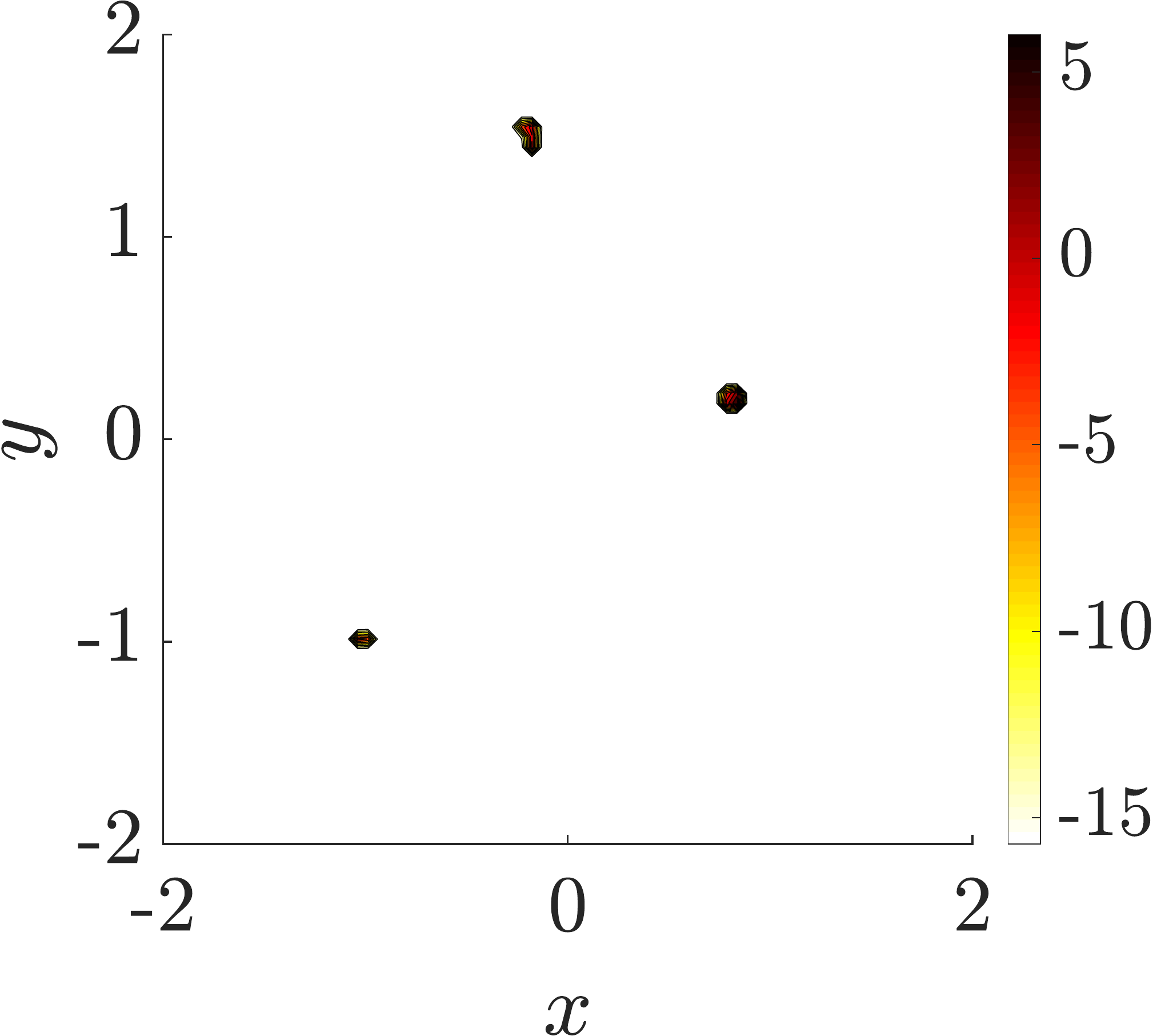}\label{fig:th_3points_errorDSM}}
\caption{Reconstructions of three point scatterers.}\label{fig:th_3points}
\end{figure}

We next demonstrate the performence of the sparse-DSM with indicator funcitonal $\mathbf{I}^{(0)}(z)$ under limited aperture measurements in Figure \ref{fig:th_la_5points}. The theoretical justification provided in Section \ref{sec:inverseScattering} does not apply in this simulation because the measurement surfaces are open curves. Indeed, while the sparse-DSM performs well in this experiment, there are \textit{false negative} reconstructions, which would not be expected if the theory developed above held fully. The are five small scatterers of different sizes and shapes to be reconstructed in this example are indicated by solid black lines. Here we take SNR=100 and 10 measurements. In Figure \ref{fig:th_la_5points_half}, measurements are taken on the half-circle of radius $4$ centered at zero below the $x$-axis. In Figure \ref{fig:th_la_5points_quarter}, measurements are taken on the quarter-circle of radius $4$ whose arc is centered on $(0,-4)$. We see that in the half-circle measurement case, all 5 obstacles are identified. The reconstruction is worse for the quarter-circle measurement case. Nonetheless, 2 scatterers are robustly identified and smeared versions of two others are reconstructed. One scatterer is missed completely. Considering the difficulty of most qualitative methods to reconstruct scatterers in moderate-noise limited-aperture scenarios, we see even the quarter-circle reconstruction as a success. 

\begin{figure}[tbhp]
\centering
\subfloat[$k$-DSM with half-circle aperture, $k_{\text{DSM}}=6$]{\includegraphics[height=1.5in]{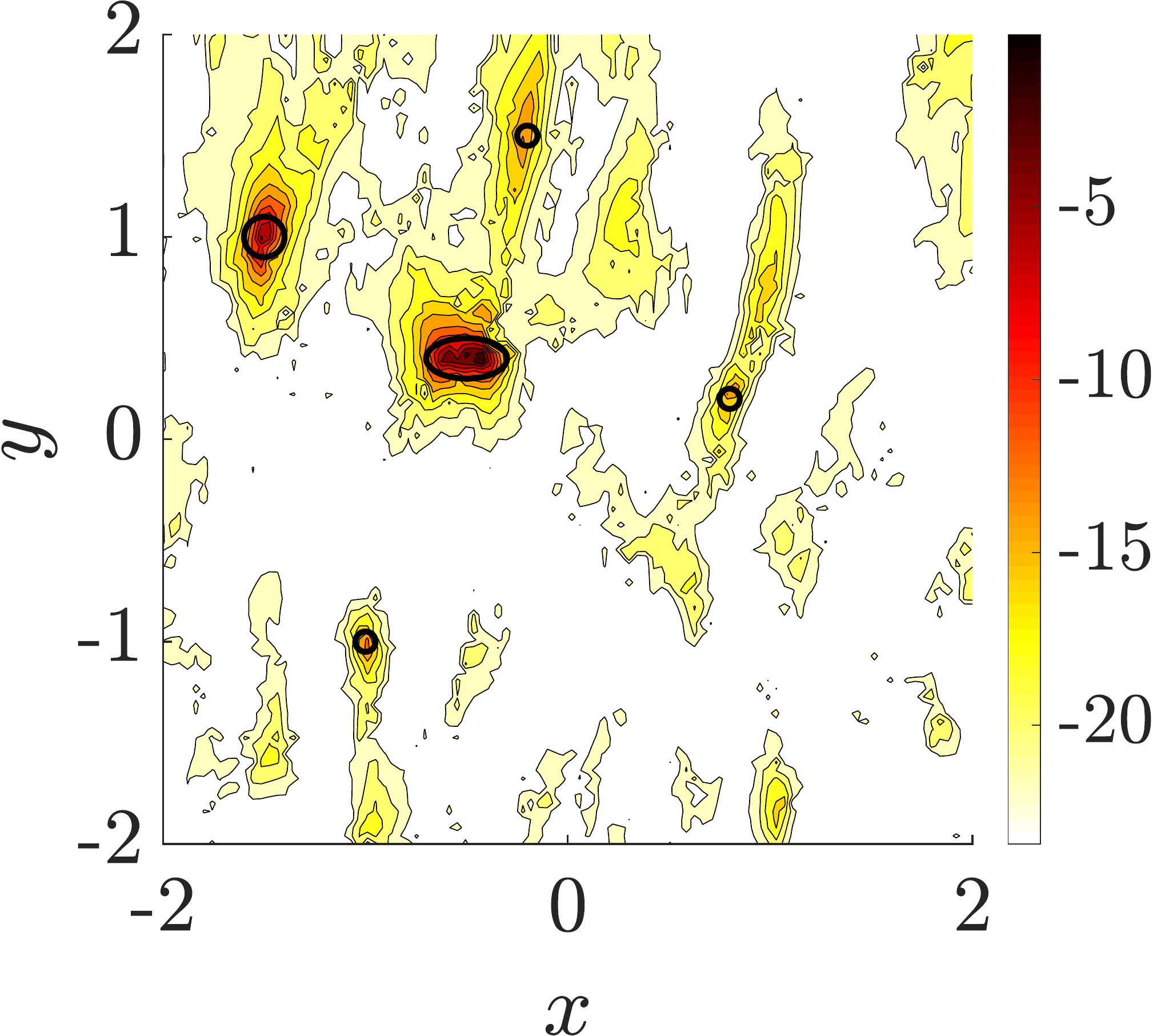}\label{fig:th_la_5points_half}}\qquad\qquad\qquad
\subfloat[$k$-DSM with quarter-circle aperture, $k_{\text{DSM}}=5$]{\includegraphics[height=1.5in]{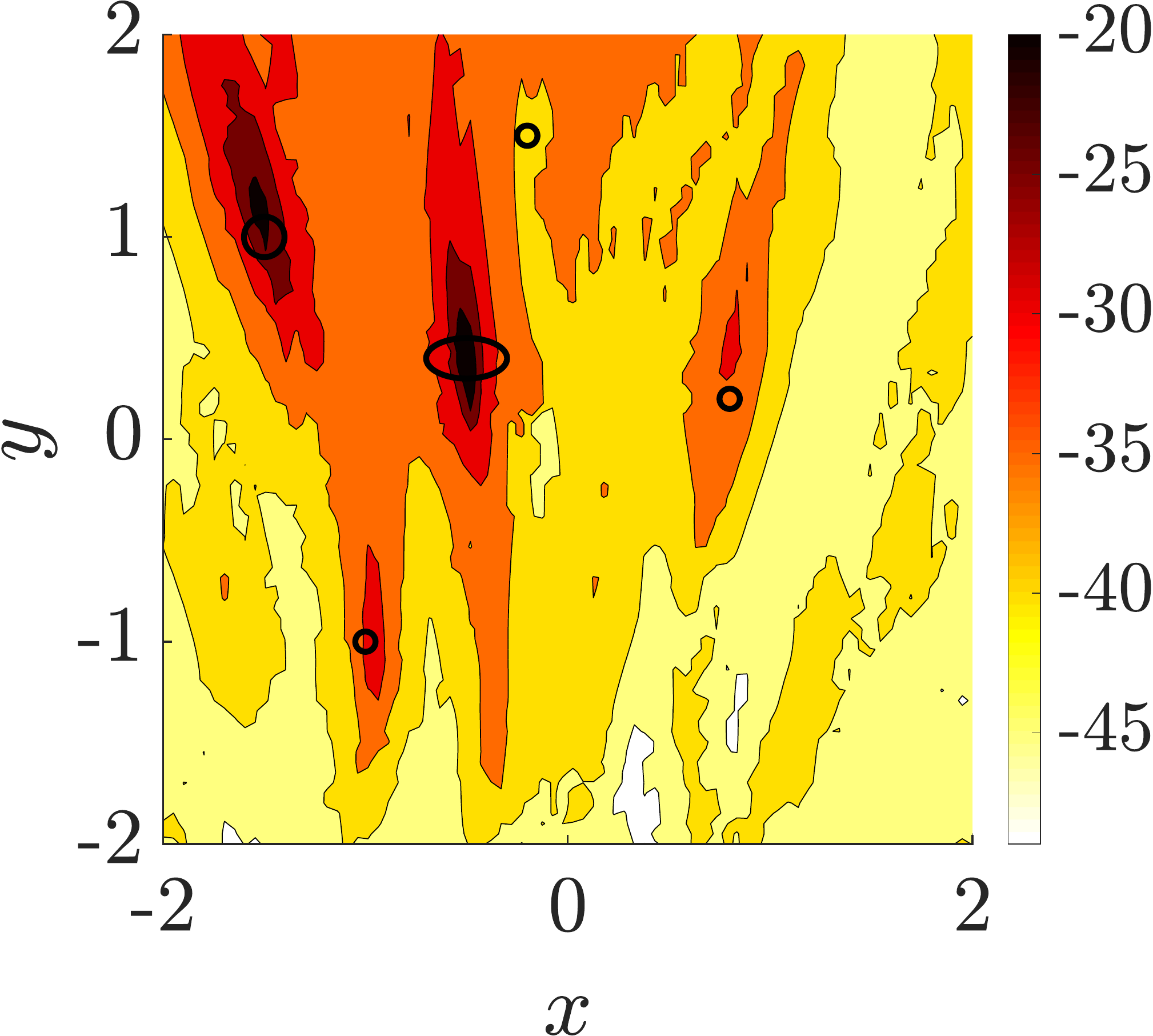}\label{fig:th_la_5points_quarter}}\\
\caption{Reconstructions of five small scatterers with a limited aperture measurement geometry.}\label{fig:th_la_5points}
\end{figure}

Finally, we turn to a case with scatterers of significantly-different sizes with the indicator function $\mathbf{I}^{(1)}$. Figures \ref{fig:weirdShape:DSM}-\ref{fig:weirdShape:errorDSM} show reconstructions of the shape in Figure \ref{fig:weirdShape:shape} from 40 multistatic transceiver with $SNR=1000$ and $k=12$. The indicator function was evaluated on $\mathcal{Z}=[-2,2]\times[-2,2]$, but we focus on $[-0.75,1.2]\times[-0.5,1]$ to show more detail. Both sparse-DSM algorithms are able to capture the multi-scale behavior of the object while the standard DSM misses the smaller circle.

\begin{figure}[tbhp]
    \centering
    \subfloat[True locations]{\includegraphics[height=1.5in]{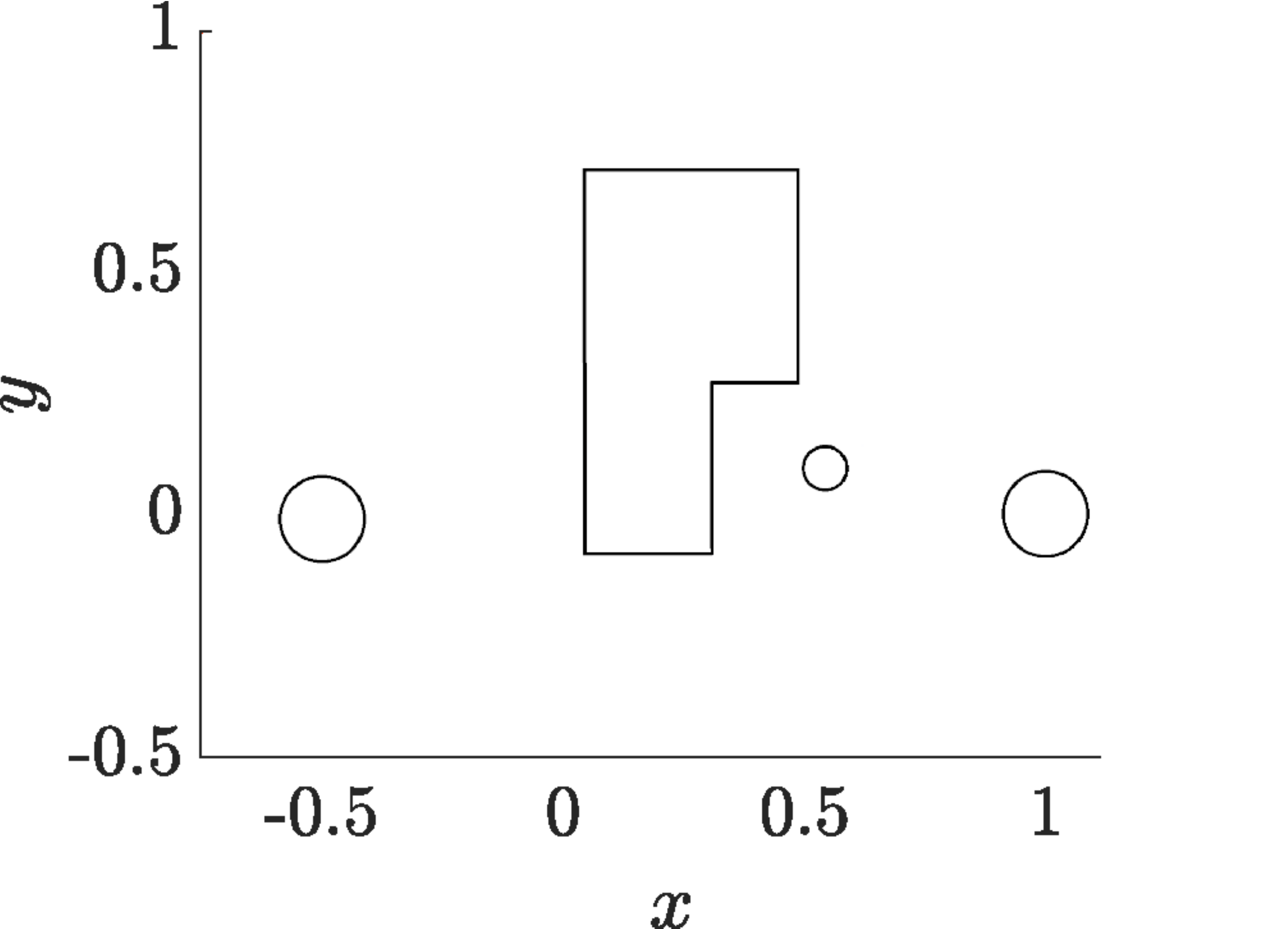}\label{fig:weirdShape:shape}}\qquad\qquad\qquad
    \subfloat[DSM]{\includegraphics[height=1.5in]{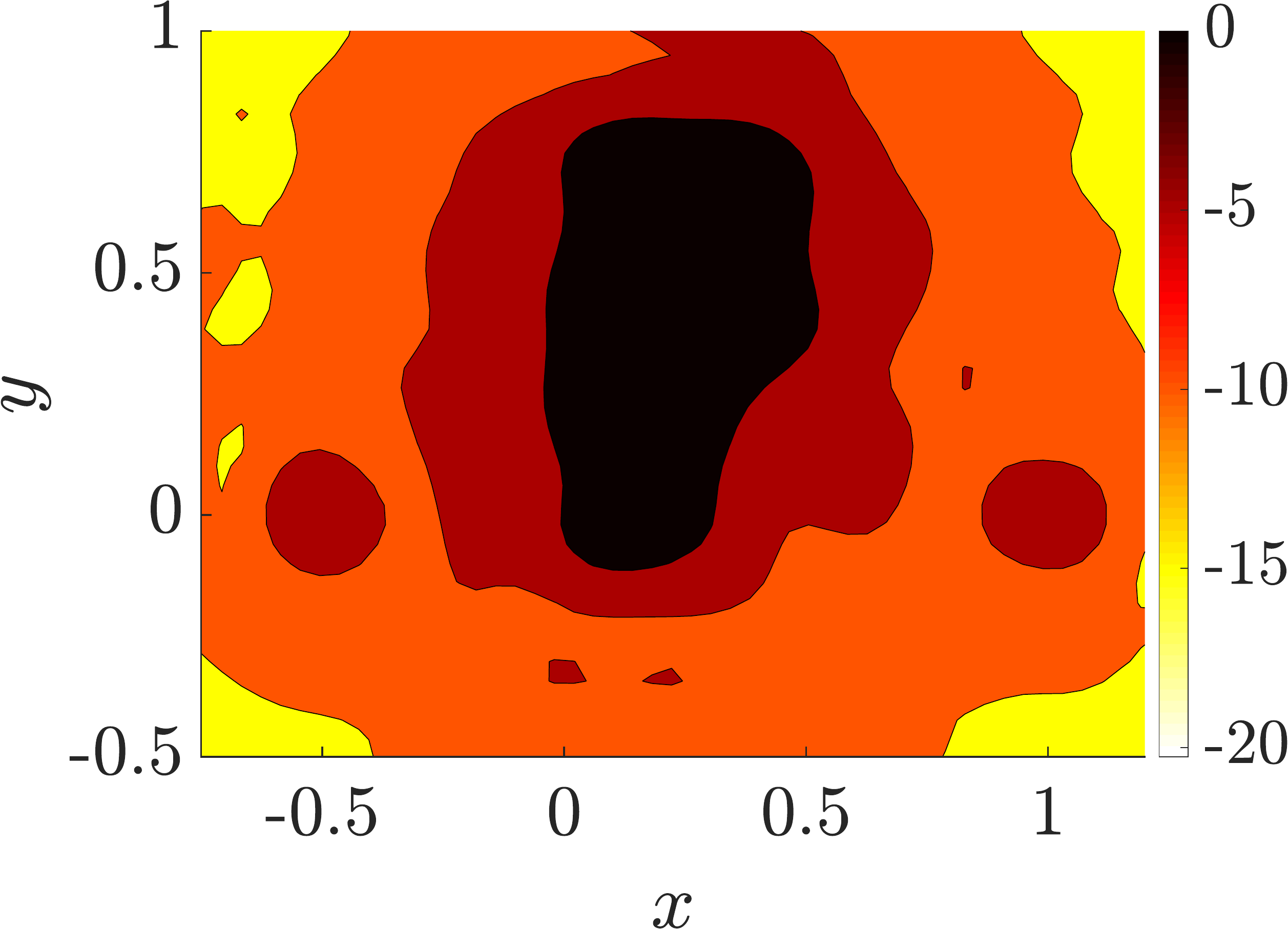}\label{fig:weirdShape:DSM}}\\
    \subfloat[$k$-DSM, $k_{\text{DSM}}=4$]{\includegraphics[height=1.5in]{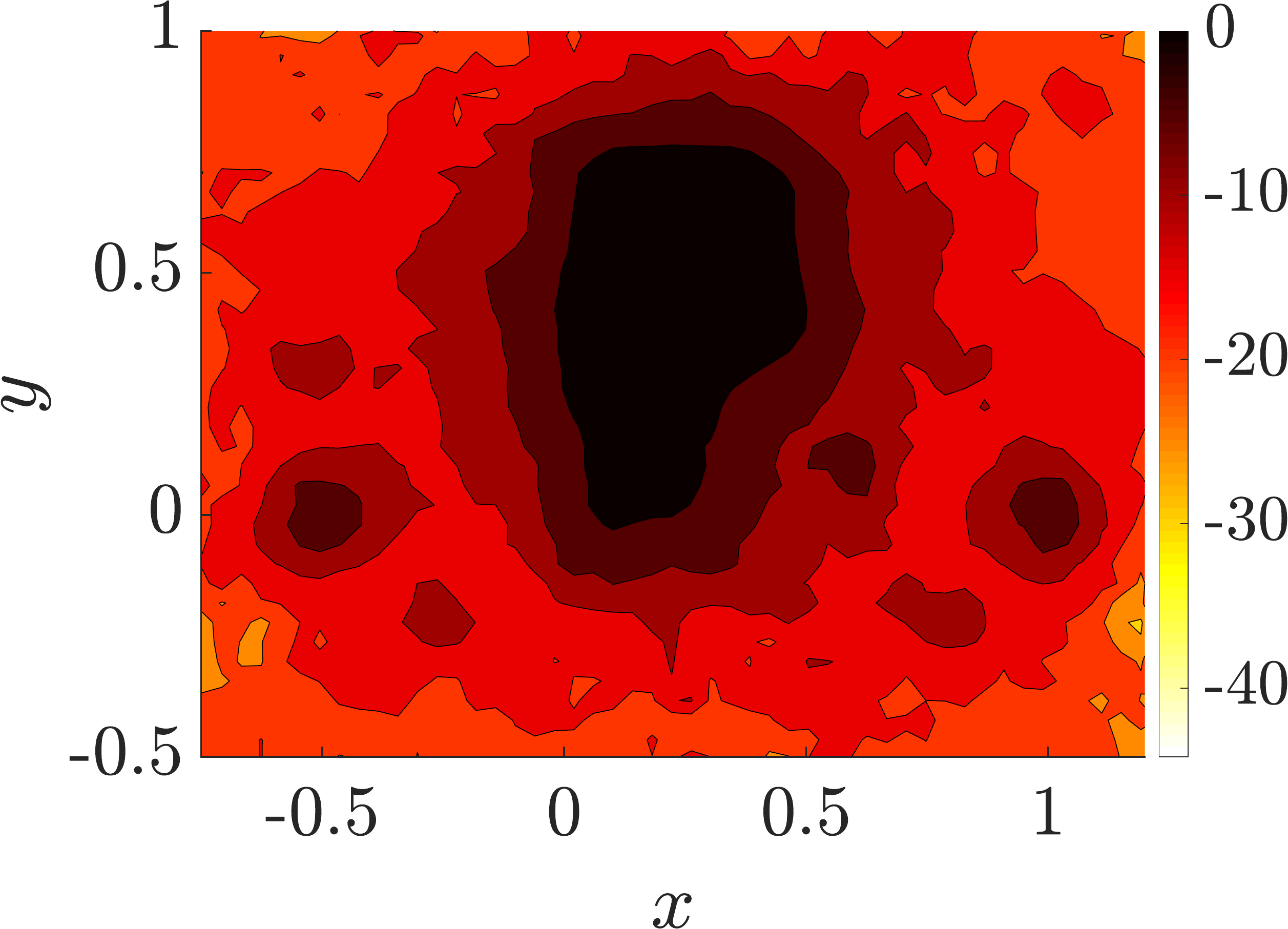}\label{fig:weirdShape:KDSM}}\qquad\qquad\qquad
    \subfloat[error-DSM, $\delta_{\text{DSM}}=0.6$]{\includegraphics[height=1.5in]{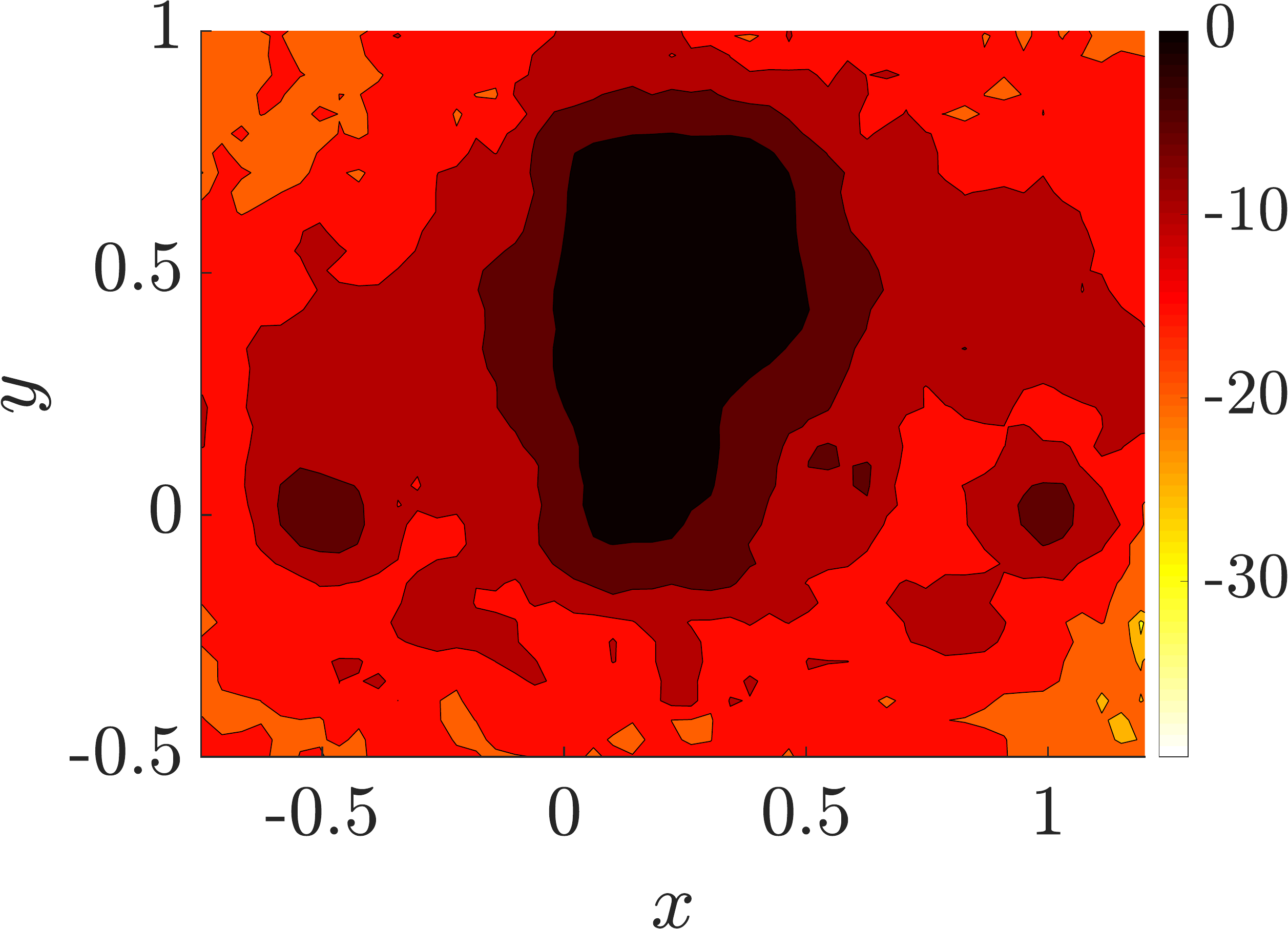}\label{fig:weirdShape:errorDSM}}
    \caption{Reconstructions of a complex geometry with multiscale obstacles.}\label{fig:weirdShape}
\end{figure}

\section{Conclusions}\label{sec:conclusion}
We have presented a new technique for estimating unknown parameters in physically-important problems without the need of a computationally-expensive forward simulation. 
The main benefits of the new technique are 1) that it is flexible with respect to application, particularly in problems relating to inversion of propagating waves in an unknown medium; 
2) it produces high-resolution estimates even for challenging situations with low-diversity of measurements; and 3) it provides a link between two inversion techniques, one with a solid theoretical background but difficult implementation and one with less theory but better practical benefits. In the future, we aim to develop a data-driven selection criterion for the sparsity parameters $k_{\text{DSM}}$ and $\delta_{\text{DSM}}$. We also plan to apply the sparse-DSM to more problems in inverse scattering, such as time-domain inversion.

There is a connection between inversion techniques for the AOA problem and the inverse scattering problem: indeed, as discussed above, this connection 
relates Capon's method and beamforming in AOA estimation to the inf-criterion and DSM in inverse scattering. The similarities between MUSIC and the linear sampling and factorization methods have been discussed in previous literature as well \cite{Cheney2001, Devaney2000, Kirsch2002}. The interplay between these problems presents an opportunity for researchers in either field to make contributions to the other, and we recommend researchers exploit these similarities to improve estimation techniques across the board.

\section*{Acknowledgments}
Thanks to B.F. Jamroz, R.W. Leonhardt, J.T. Quimby, K.A. Remley, P.G. Vouras, A.J. Weiss, and D.F. Williams for providing the data and pre-processing code for the experiments in Section \ref{sec:measuredAOA} and for providing insight on this problem. Thanks to A. Dienstfrey, A. Feldman, P.D. Hale, B.F. Jamroz, and D.F. Williams for a careful reading of this manuscript.


\begin{thebibliography}{199}

\bibitem{AndrewsEtAl2014} J.G. Andrews, S. Buzzi, W. Choi, S.V. Hanly, A. Lozano, A.C.K. Soong, and J.C. Zhang, \textit{What Will 5G Be?}, IEEE J. Sel. Area Comm., 32 (2014), pp. 1065-1082.

\bibitem{AlqadahEtAl2011} H.F. Alqadah, M. Ferrara, H. Fan, and J.T. Parker, \textit{Single frequency inverse obstacle scattering: a sparsity constrained linear sampling method approach}, IEEE Trans. Image Process., 21 (2011), pp. 2062-2074.

\bibitem{AlqadahEtAl2011b} H.F. Alqadah, J. Parker, M. Ferrara, and H. Fan, \textit{Space-frequency sparse regularization for the linear sampling method}, 13th Proceedings of the International Conference on Electromagnetics in Advanced Applications, Torino, Italy, 2011, pp. 421-424.

\bibitem{AlqadahValdivia2013} H.F. Alqadah and N. Valdivia, \textit{A frequency based constraint for a multi-frequency linear sampling method}, Inverse Problems, 29 (2013), 095019.

\bibitem{Alqadah2016} H.F. Alqadah, \textit{A compressive multi-frequency linear sampling method for underwater acoustic imaging}, IEEE Trans. on Image Process., 25 (2016), pp. 2444-2455. 

\bibitem{AmariEtAl2013} H. Ammari, J. Garnier, W. Jing, H. Kang, M. Lim, K. S\"{o}lna, and H. Wang, \textit{Mathematical and Statistical Methods for Multistatic Imaging}, Lecture Notes in Math. 2098, Springer-Verlag, 2013.

\bibitem{Arens2011} T. Arens, D. Gintides, and A. Lechleiter, \textit{Direct and inverse medium scattering in a three-dimensional homogeneous planar waveguide}, SIAM J. Appl. Math., 71 (2011), pp. 753-772. 

\bibitem{ArensJiLiu2020} T. Arens, X. Ji, and X. Liu, \textit{Inverse electromagnetic obstacle scattering problems with multi-frequency sparse backscattering far field data}, ArXiv Preprint 2002.09163v1 (2020). 

\bibitem{ArensKirsch2003} T. Arens and A. Kirsch, \textit{The factorization method in inverse scattering from periodic structures}, Inverse Problems, 19 (2003), pp. 1195–1211

\bibitem{AudibertHaddar2017} L. Audibert and H. Haddar, \textit{The generalized linear sampling method for limited aperture measurements}, SIAM J. Imaging Sci., 10 (2017), pp. 845-870.

\bibitem{Bartlett1948} M.S. Bartlett, \textit{Smoothing periodograms from time-series with continuous spectra}, Nature, 161 (1948), pp. 686-687.

\bibitem{BorceaMeng2019} L. Borcea and S. Meng, \textit{Factorization method versus migration imaging in a waveguide}, Inverse Problems, 35 (2019), 124006.

\bibitem{BourgeoisLuneville2013} L. Bourgeois and E. Luneville, \textit{On the use of the linear sampling method to identify cracks in elastic waveguides}, Inverse Problems, 29 (2013), 025017. 

\bibitem{BoydVandenberghe2004} S.P. Boyd and L. Vandenberghe, \textit{Convex Optimization}, Cambrige University Press, Cambridge, 2004. 

\bibitem{CakoniColton2014} F. Cakoni and D. Colton, \textit{A Qualitative Approach to Inverse Scattering Theory}. Springer, New York 2014. 

\bibitem{CakoniColtonHaddar2016} F. Cakoni, D. Colton, and H. Haddar, \textit{Inverse Scattering Theory and Transmission Eigenvalues}, CBMS-NSF Regional Conf. Ser. in Appl. Math. 88, SIAM, Philadelphia 2016. 

\bibitem{CakoniColtonMonk2011} F. Cakoni, D. Colton, and P. Monk, \textit{The Linear Sampling Method in Inverse Electromagnetic Scattering}, CBMS-NSF Regional Conf. Ser. in Appl. Math. 80, SIAM, Philadelphia 2011. 

\bibitem{CakoniRezac2017} F. Cakoni and J.D. Rezac, \textit{Direct imaging of small scatterers using reduced time dependent data}, J. Comput. Phys., 338 (2017), pp. 371-387.

\bibitem{Capon1969} J. Capon, \textit{High-resolution frequency-wavenumber spectrum analysis}, Proc. IEEE, 57 (1969), pp. 1408-1418.

\bibitem{ChaiMoscosoPapanicolaou2013} A Chai, M. Moscoso, and G. Papanicolaou, \textit{Robust imaging of localized scatterers using the singular value decomposition and $\ell_1$ minimization}, Inverse Problems, 29 (2013), 025016. 

\bibitem{ChaiMoscosoPapanicolaou2014} A Chai, M. Moscoso, and G. Papanicolaou, \textit{Imaging strong localized scatterers with sparsity promoting optimization}, SIAM J. Imaging Sci., 7, (2014), pp. 1358-1387. 

\bibitem{CharalambopoulosEtAl2003} A. Charalambopoulos, D. Gintides, and K. Kiriaki, \textit{The linear sampling method for non-absorbing penetrable elastic bodies}, Inverse Problems, 19 (2003), pp. 549–561.

\bibitem{CharalambopoulosEtAl2007} A. Charalambopoulos, A. Kirsch, K.A. Anagnostopoulos, D. Gintides and K. Kiriaki, \textit{The factorization method in inverse elastic scattering from penetrable bodies}, Inverse Problems, 23 (2007), pp. 27–51.

\bibitem{Cheney2001} M. Cheney, \textit{The linear sampling method and the MUSIC algorithm}, Inverse Problems, 17 (2001), pp. 591-595. 

\bibitem{ColtonKirsch1996} D. Colton and A. Kirsch, \textit{A simple method for solving inverse scattering problems in the resonance region}, Inverse Problems, 12 (1996), pp. 383-393.

\bibitem{ColtonKress2013} D. Colton and R. Kress, \textit{Inverse Acoustic and Electromagnetic Scattering Theory}, 3rd Edition, Springer, New York 2013.

\bibitem{ChenChenHuang2013} J. Chen, Z. Chen and G. Huang, \textit{Reverse time migration for extended obstacles: acoustic waves}, Inverse Problems, 29 (2013), 085005.

\bibitem{Devaney2000} A.J. Devaney, \textit{Super-resolution processing of multi-static data using time reversal and MUSIC}, Unpublished paper, preprint available on the author’s website (2000). 

\bibitem{Fannjiang2010} A.C. Fannjiang, \textit{The MUSIC algorithm for sparse objects: a compressed sensing analysis}, Inverse Problems, 27 (2011), 035013. 

\bibitem{HarrisNguyen2019} I. Harris and D.-L. Nguyen, \textit{Orthogonality Sampling Method for the Electromagnetic Inverse Scattering Problem} SIAM J. Sci. Comput., 42 (2020), pp. B722–B737. 

\bibitem{HarrisRome2017} I. Harris and S. Rome, \textit{Near field imaging of small isotropic and extended anisotropic scatterers}, Appl. Anal., 96 (2017), pp. 1713-1736. 

\bibitem{Hokanson2013} J.M. Hokanson, \textit{Numerically Stable and Statistically Efficient Algorithms for Large Scale Exponential Fitting}, Ph.D thesis, Rice University, Houston, TX, 2013. 

\bibitem{HuEtAl2014} G. Hu, J. Yang, B. Zhang, and H. Zhang, \textit{Near-field imaging of scattering obstacles with the factorization method}, Inverse Problems 30 (2014), 095005. 

\bibitem{Ito2012} K. Ito, B. Jin, and J. Zou, \textit{A direct sampling method to an inverse medium scattering problem}, Inverse Problems, 28 (2012), 025003. 

\bibitem{Ito2013} K. Ito, B. Jin, and J. Zou, \textit{A direct sampling method for inverse electromagnetic medium scattering} Inverse Problems, 29 (2013), 095018.

\bibitem{KimLeeYe2012} J.M. Kim, O.K. Lee, and J.C. Ye, \textit{Compressive MUSIC: Revisiting the link between compressive sensing and array signal processing}, IEEE Trans. Inform. Theory, 58 (2012), pp. 278-301.

\bibitem{Kirsch1998} A. Kirsch, \textit{Characterization of the shape of a scattering obstacle using the spectral data of the far field operator}, Inverse Problems, 14 (1998), pp. 1489–1512. 

\bibitem{Kirsch2002} A. Kirsch, \textit{The MUSIC algorithym and the factorization method in inverse scattering theory for inhomogeneous media}, Inverse Problems, 28 (2002), pp. 1025-1040. 

\bibitem{Kirsch2007} A. Kirsch, \textit{An integral equation for Maxwell's equations in a layered medium with an application to the factorization method}, J. Integral Equations Appl., 19 (2007), pp. 333-357. 

\bibitem{LeemLiuPelekanos2018} K.H. Leem, J. Liu, and G. Pelekanos, \textit{Two direct factorization methods for inverse scattering problems}, Inverse Problems, 34 (2018), 125004. 

\bibitem{Liu2017} X. Liu, \textit{A novel sampling method for multiple multiscale targets from scattering amplitudes at a fixed frequency}, Inverse Problems, 33 (2017), 085011. 

\bibitem{KirschGrinberg2008} A. Kirsch and N. Grinberg, \textit{The Factorization Method for Inverse Problems}, Oxford University Press, Oxford 2008. 

\bibitem{LeeBresler2012} K. Lee and Y. Bresler, \textit{Subspace methods for joint sparse recovery}, IEEE Trans. Inform. Theory, 58 (2012), pp. 3613-3641.

\bibitem{MalioutovCetinWillsky2005} D. Malioutov, M. \c{C}etin, and A.S. Willsky, \textit{A sparse signal reconstruction perspective for source localization with sensor arrays}, IEEE Trans. on Signal Process., 53 (2005), pp. 3010-3022.

\bibitem{Nguyen2019} D.L. Nguyen, \textit{Direct and inverse electromagnetic scattering problems for bi-anisotropic media}, Inverse Problems, 35 (2019), 124001.

\bibitem{PillaiKwon1989} S.U. Pillai and B.H. Kown, \textit{Forward/backward spatial smoothing technique for coherent signal identification}, IEEE Trans. Acoust. Speech Signal Process., 37 (1989), pp. 8-15.

\bibitem{PeterPlonka2013} T. Peter and G. Plonka, \textit{A generalized Prony method for reconstruction of sparse sums of eigenfunctions of linear operators}, Inverse Problems, 29 (2013), 025001. 

\bibitem{PourahmadianGuzinaHaddar2017} F. Pourahmadian, B.B. Guzina, and H. Haddar, \textit{Generalized linear sampling method for elastic-wave sensing of heterogeneous fractures}, Inverse Problems, 33 (2017), 055007.

\bibitem{Potthast2010} R. Potthast, \textit{A study on orthogonality sampling}, Inverse Problems, 26 (2010), 074015.

\bibitem{RanganEtAl2014} S. Rangan, T.S. Rappaport, and E. Erkip, \textit{Millimeter-wave cellular wireless networks: potentials and challenges.}, Proc. IEEE, 102 (2014), pp. 366-385.

\bibitem{Rezac2017} J.D. Rezac, \textit{Direct Methods for Inverse Scattering with Time Dependent and Reduced Data}. Ph.D Thesis, University of Delaware, Newark, DE, 2017. 

\bibitem{RohEtAl2014} W. Roh, J-Y Seol, J. Park, B. Lee, J. Lee, Y. Kim, J. Cho, K. Cheun, and F. Aryanfar, \textit{Millimeter-wave beamforming as an enabling technology for 5G cellular communications: theoretical feasibility and prototype results}, IEEE Commun. Mag., 52 (2014), pp. 106-113. 

\bibitem{RostThomas2002} S. Rost and C. Thomas, \textit{Array seismology: methods and applications}, Rev. Geophys., 40 (2002), pp. 2-27.

\bibitem{RubinsteinEtAl2008} R. Rubinstein, M. Zibulevsky, and M. Elad, \textit{Efficient implementation of the K-SVD algorithm using batch orthogonal matching pursuit}, Technical Report - CS Technion (2008).

\bibitem{Schmidt1986} R.O. Schmidt, \textit{Multiple emitter location and signal parameter estimation}, IEEE Trans. Ant. Prop., 34 (1986), pp. 276-280.

\bibitem{ShanWaxKailath1985} T.-J. Shan, M. Wax, and T. Kailath, \textit{On spatial smoothing for direction-of-arrival estimation of coherent signals}, IEEE Trans. Acoust. Speech Signal Process, 33 (1985), pp. 806-811. 

\bibitem{VanTrees2002} H.L Van Trees, \textit{Optimum Array Processing: Detection, Estimation and Modulation Theory (Part IV)}. Wiley, New York 2002. 

\bibitem{VanVeenBuckley1988} B.D. Van Veen and K.M Buckley, \textit{Beamforming: a versatile approach to spatial filtering}, IEEE Acoust. Speech Signal Process. Mag., 5 (1988), pp. 4-24. 

\bibitem{WeissEtAl2019} A. J. Weiss, J.T. Quimby, R.W. Leonhardt, B.F. Jamroz, D.F. Williams, K.A. Remley, P.G. Vouras, and A. Elsherbeni, \textit{Setup and Control of a Millimeter-Wave Synthetic Aperture Measurement System with Uncertainties}, Proceedings of the 95th ARFTG Microwave Measurement Conference (to appear). 

\bibitem{ZiskindWax1988} I. Ziskind and M. Wax, \textit{Maximum likelihood localization of multiple sources by alternating projection}, IEEE Trans. Acoust. Speech Signal Process., 36 (1988), pp. 1553-1560.

\end{thebibliography}
\end{document}